\documentclass[12pt]{article}
	
	\usepackage[margin=1in]{geometry}  
	\usepackage{graphicx}              
	\usepackage{amsmath}               
	\usepackage{amsfonts}              
	\usepackage{amsthm}                
	\usepackage{mathtools}
	\usepackage{xcolor}
	\usepackage{hyperref}
	\usepackage{tikz-cd}
    \usepackage{amssymb}
	
	\newtheorem{thm}{Theorem}[section]
	\newtheorem{lem}[thm]{Lemma}
	\newtheorem{prop}[thm]{Proposition}
	\newtheorem{cor}[thm]{Corollary}
	
	\newtheorem{defn}[thm]{Definition}
	
	\newtheorem{remark}[thm]{Remark}

\newcommand{\HH}{\mathcal H}
\newcommand{\RR}{\mathbb R}
\newcommand{\ddc}{\sqrt{-1}\partial\bar\partial}
\newcommand{\idealprec}{\mathrel{\prec\!\prec}}
\newcommand{\Ga}{\mathrm{G}}
\newcommand{\IG}{\mathrm{I}}
\newcommand{\Pol}{\operatorname{Pol}}
\newcommand{\Herm}{\mathrm{Herm}}
\newcommand{\Psh}{\mathrm{Psh}}

\begin{document}

\title{The coerciveness and generalised Monge-Amp\`{e}re equations}

\author{Gao Chen, Kartick Ghosh, and Ziyue Wang}
\maketitle

\begin{abstract}
  In this paper, we prove the equivalence between coerciveness and the existence of a solution to the generalised Monge-Amp\`{e}re equations for right-Noetherian polynomials satisfying several assumptions. This includes the case of interlacing polynomials in strict proper position. In particular, we have generalised Chu-Lee's result to supercritical but not hypercritical LYZ equations.
\end{abstract}

\tableofcontents 

\section{Introduction}

The search for a canonical metric in a K\"ahler class is a central problem in complex differential geometry. The most important one is a constant scalar curvature K\"ahler (cscK) metric. The K\"ahler-Ricci-flat case was solved by Yau in \cite{YauCalabiConjecture} using the complex Monge-Amp\`ere equation.

From the variational point of view, cscK metrics are critical points of Mabuchi's K-energy \cite{MabuchiKEnergy}. It is convex along smooth geodesics. Chen's construction of weak geodesics \cite{ChenSpaceKahlerMetrics} and Darvas's study of the $\mathcal{E}^p$ space made it possible to conjecture that the existence of cscK metrics is equivalent to the coerciveness of the K-energy with respect to the $d_1$-distance after taking account of the automorphism group. This conjecture was proved by Berman-Darvas-Lu \cite{BermanDarvasLuRegularity} and Chen-Cheng \cite{ChenChengCSCKEstimates,ChenChengCSCKExistence} in two directions. Based on these results, on projective manifolds, recently Trusiani \cite{TrusianiCscK}, with an appendix by Boucksom, proved the equivalence of the existence of a cscK metric with a uniform algebraic criterion: the uniform K-stability.

By writing the K-energy as the sum of entropy and the J-functional, Chen introduced the J-functional \cite{ChenJFlow}, whose critical point is the J-equation proposed by Donaldson \cite{DonaldsonMomentMaps} from a different motivation. Using Song-Weinkove's analytic results \cite{SongWeinkoveJFlow}, Collins-Sz\'ekelyhidi \cite{CollinsSzekelyhidiJFlow} proved the equivalence between the solvability of the J-equation and the coerciveness of the J-functional. The algebraic criterion was proved by Collins-Sz\'ekelyhidi \cite{CollinsSzekelyhidiJFlow} in the toric case, and by the first author in the K\"ahler case \cite{ChenJEquationDHYM}, with an improvement by Song \cite{SongNakaiMoishezon} addressing the uniform issue.

The J-equation is also the large radius limit of the LYZ equation posed by Leung-Yau-Zaslow \cite{LeungYauZaslowFourierMukai} motivated by the study of the special Lagrangian equation using mirror symmetry. Based on Collins-Jacob-Yau's \cite{CollinsJacobYauLagrangianPhase} analytic results, the first author \cite{ChenJEquationDHYM} also proves the equivalence of the solvability of the supercritical LYZ equation with a uniform algebraic criterion along a test family. Later, Chu-Lee-Takahashi \cite{ChuLeeTakahashiNakaiMoishezon} relaxed the uniform condition in the K\"ahler case, and removed the test family in the projective case. However, the test family can't be removed in the K\"ahler but not projective case by the counterexample of Zhang \cite{ZhangSupercriticalDHYM}. This makes the algebraic criterion hard to study. Nevertheless, Chu-Lee \cite{ChuLeeHypercriticalDHYM} proved that, based on the geodesic theory provided by Collins-Yau \cite{CollinsYauGeodesics}, the solvability of the hypercritical LYZ equation is equivalent to the coerciveness when the complex dimension is at least 4. 

The supercritical but not hypercritical case of the LYZ equation motivates the study of our paper. However, we would like to study it in a more general framework introduced by Lin \cite{LinInverseSigmaConvexity,LinInverseSigmaConvexityCorrigendum} called a right-Noetherian polynomial. This has been studied in more detail and generalised to more general settings by Fang-Ma \cite{FangMaGarding,FangMaIdealGarding}. In the projective case, the algebraic criterion has been proved recently by the first author and two collaborators Nie and Xu \cite{ChenNieXuNumericalCriterion}, who used the analytic tool by Lin \cite{LinInverseSigmaSolvability} and extended an earlier work by Datar-Pingali \cite{DatarPingaliNumericalCriterion}. In the K\"ahler but not projective case, the first and second authors and Nie \cite{ChenGhoshNieGeodesics} proved the existence of weak geodesics, and defined the $d_p$ distance. 

In this paper, we prove an existence-coerciveness equivalence in the same setting as \cite{ChenGhoshNieGeodesics} with several additional assumptions. More precisely, let $(X,\omega)$ be a compact connected K\"ahler manifold of complex dimension $n\ge 2$ since the $n=1$ case is trivial to study. Let $\chi$ be a smooth real closed $(1,1)$-form, and $f_1$ and $f_2$ be right-Noetherian polynomials with degree $n$. Then we define
\[
 f_i(x)=\sum_{k=0}^n c_k^{(i)}\binom nk x^k,
 \qquad g_i(\lambda)=\sum_{k=0}^n c_k^{(i)}\sigma_k(\lambda),
 \qquad c_n^{(i)}=1, \qquad c_{n-1}^{(2)}=-1,
\]
and
\[
 g_i(\chi_\varphi)=\sum_{k=0}^n c_k^{(i)}\binom nk
       \chi_\varphi^k\wedge\omega^{n-k},
 \qquad \chi_\varphi=\chi+\ddc\varphi,\quad i=1,2.
\]
The conditions $c_n^{(i)}=1$ and $c_{n-1}^{(2)}=-1$ are not essential, and can be achieved by normalizations. We may omit the normalization process and study the general $c_{n-1}^{(2)}$ and $c_n^{(i)}>0$ in this paper to simplify the notation.

Define
\[
 \HH=\{\varphi\in C^\infty(X,\RR):
           \chi_\varphi\in\Upsilon_{g_2}(\omega)\},
\] and the $d_p$, $p\in[1,\infty)$, distance on it by
\[
 d_p(\varphi_0,\varphi_1)
       =\inf_{\substack{\varphi_t\in\HH\\
              \varphi_t|_{t=0}=\varphi_0,\ \varphi_t|_{t=1}=\varphi_1}}
          \int_0^1\Big(\int_X|\frac{d\varphi_t}{dt}|^p g_2(\chi_{\varphi_t})\Big)^{1/p}\,dt.
\]
Without loss of generality, we assume that $0\in\HH$. For $i=1,2$, the functionals $I_{i}$ on $\HH$ are defined by 
\[
I_{i}(\varphi)=\int_{X}\varphi\sum_{k=0}^{n}\binom{n}{k}c_{k}^{(i)}\sum_{j=0}^{k}\frac{1}{j+1}\binom{k}{j}\chi^{k-j}\wedge(\sqrt{-1}\partial\bar{\partial}\varphi)^{j}\wedge\omega^{n-k},
\]
$J_i=-I_i$, and they satisfy
\begin{equation}
    \label{functional properties}
    \frac{d}{dt}I_{i}(\varphi_t)=\int_{X}\frac{d\varphi_t}{dt}g_{i}(\chi_{\varphi_t}).
\end{equation}
Following the idea of Collins-Sz\'ekelyhidi \cite{CollinsSzekelyhidiJFlow}, Fu-Yau-Zhang \cite{FuYauZhangFlow}, and Chu-Lee \cite{ChuLeeHypercriticalDHYM}, we use the flow
\begin{equation}\label{eq:flow}
 \frac{\partial \varphi}{\partial t}=F_\omega(\chi_\varphi),\qquad
 F_\omega(\chi_\varphi):=\frac{g_1(\chi_\varphi)}{g_2(\chi_\varphi)},
 \qquad\varphi(0)=\varphi_0\in\HH
\end{equation}
to study the coerciveness problem.

To make sure that the flow is parabolic, we strengthen the proper condition assumption $f_2\prec f_1$ in \cite{FangMaGarding} to a strongly strict version. See Definition \ref{SSPrec} for more details. To apply the Evans-Krylov estimate, we also require the concavity of $\frac{g_1}{g_2}$ on $\mathcal{H}$, which requires that $f_2\idealprec f_1$ in the sense of Definition 9.4 of \cite{FangMaIdealGarding}.

We define the cohomological constant $V_i:=\int_X g_i(\chi)$, and $v_i=\frac{V_i}{\int_X \omega^n}$. Then our main theorem is the following:

\begin{thm}\label{thm:main}
Assume $f_1, f_2$ are right-Noetherian, $f_2\prec f_1$ strongly strictly, $f_2\idealprec f_1$, and $V_1=0$. Then the following are equivalent:
\begin{enumerate}
\item The equation $g_1(\chi_\varphi)=0$ has a solution
      $\varphi\in\HH$ with $I_2(\varphi)=0$.
\item There exist $\delta>0$ and $C$ such that
\[
 J_1(\varphi)\ge\delta d_1(0,\varphi)-C
 \quad\text{for all }\varphi\in\HH\text{ with }I_2(\varphi)=0.
\]
\end{enumerate}
\end{thm}

We briefly describe the proof. The strongly strict proper position condition and the ideal proper position condition ensure the long-time existence of the flow~\eqref{eq:flow}. Under coerciveness, a subsequence of the flow converges to a current satisfying the degenerate cone condition, and has zero Lelong number by \cite{GuedjZeriahiFiniteEnergy}. A compatible perturbation of both polynomials provides the strict positivity needed to regularize this limit current. Combining B{\l}ocki and Ko{\l}odziej's local regularization
\cite{BlockiKolodziejRegularization} with Lin's solvability theorem
\cite{LinInverseSigmaSolvability} then gives a smooth solution.

For the converse, we take a solution as the reference potential and
introduce an analogue of Aubin's $I-J$ functional called $\mathcal{A}(\varphi)$. Then we use a second-variation formula on the regularized geodesic of \cite{ChenGhoshNieGeodesics} to obtain a lower bound on $J_1(\varphi)-\varepsilon \mathcal{A}(\varphi)$. This gives the coerciveness condition in Theorem \ref{thm:main}.

The paper is organized as follows. Sections 2-3 recall the
Fang-Ma-G{\aa}rding framework and study the strict version of the
proper position condition. Sections 4-5 prove the volume comparison and
construct the compatible perturbation. Sections 6-7 establish
short-time and long-time existence of the flow. The two implications
in Theorem \ref{thm:main} are proved in Sections 8 and 9, respectively. Section 10 proves that all strictly interlacing polynomials, including the supercritical LYZ equation, satisfy all the conditions.

\paragraph{\textbf{Declaration on the use of AI.}}
The decomposition of the whole paper into propositions, as well as the original version of the propositions, is due to the authors. Then we asked ChatGPT 6 Astra to produce counterexamples or proofs of special cases of each proposition. Then we found ideas to change the propositions inspired by the feedback. After several rounds of intensive discussions, we ended up with the right propositions. Then we checked and edited the proof substantially for mathematical clarity. ChatGPT 6 Astra was also used to assist with grammar correction.

\section{Preliminaries}

 This section recalls the definitions and theorems from
\cite{FangMaGarding}, \cite{FangMaIdealGarding}, and \cite{LinInverseSigmaConvexity} with a correction in \cite{LinInverseSigmaConvexityCorrigendum}.

\begin{defn}[Definition 3.1(7),(9) of \cite{FangMaGarding}]
Let \[\sigma_k(x_1,\ldots,x_n)=\sum_{i_1<\cdots<i_k}x_{i_1}\cdots x_{i_k}\] be the \(k\)-th elementary symmetric
polynomial, with \(\sigma_0=1\).

If \(h(x)=\sum_{k=0}^{d}c_k \binom nk x^k\) and \(n\ge d\), its ordinary polarization is
\[
  \Pol_n(h)(x_1,\ldots,x_n)
  =
  \sum_{k=0}^{d}c_k\sigma_k(x_1,\ldots,x_n).
\]
It is the unique symmetric multi-affine polynomial satisfying
\[
  \Pol_n(h)(x,\ldots,x)=h(x).
\]

More generally, let
\(\kappa=(\kappa_1,\ldots,\kappa_m)\in\mathbb N^m\), and
suppose \(h(x_1,\ldots,x_m)=\sum_{\alpha\le\kappa}
c_\alpha x^\alpha\), where
\(\alpha=(\alpha_1,\ldots,\alpha_m)\),
\(x^\alpha=\prod_i x_i^{\alpha_i}\), and
\(\alpha\le\kappa\) means
\(\alpha_i\le\kappa_i\) for every \(i\).  The
\(\kappa\)-polarization is
\[
  \Pi^\uparrow_{\kappa}(h)
  =
  \sum_{\alpha\le\kappa}c_\alpha
  \prod_{i=1}^m
  \frac{\sigma_{\alpha_i}(x_{i1},\ldots,x_{i\kappa_i})}
       {\binom{\kappa_i}{\alpha_i}}.
\]
For example, if $f_1(x)$, $f_2(x)$ are univariable functions with degree at most $n$, $g_1(x)$, $g_2(x)$ are their polarizations on $\mathbb{R}^n$, and $\kappa=(n,1)$, then
\[
  \Pi^\uparrow_{\kappa}(f_2(x)y+f_1(x))
  =
  g_2(x)y+g_1(x).
\]
\end{defn}

\begin{defn}[{Definitions 2.2 and 2.3} of \cite{FangMaGarding}]
For \(n\ge1\), write
\begin{align*}
  \Gamma_n^+
  &=\{x=(x_1,\ldots,x_n)\in\RR^n:x_i>0\text{ for every }i\},\\
  \overline{\Gamma_n^+}
  &=\{x=(x_1,\ldots,x_n)\in\RR^n:x_i\ge0\text{ for every }i\}.
\end{align*}
A set \(\Upsilon\subseteq\RR^n\) passes the \emph{positive ray test}
\emph{(PRT)} if
\[
  \Upsilon+\overline{\Gamma_n^+}\subseteq \Upsilon.
\]
It is \emph{negative ray terminating} \emph{(NRT)} if, for
every \(x\in \Upsilon\), the set
\[
  (\{x\}-\overline{\Gamma_n^+})\cap \Upsilon
\]
is bounded.
\end{defn}

We now recall the definition of a G{\aa}rding polynomial from \cite{FangMaGarding} and an ideal G{\aa}rding polynomial from \cite{FangMaIdealGarding}. To distinguish this notion from the classical one, we call such a polynomial a Fang-Ma-G{\aa}rding polynomial or an ideal Fang-Ma-G{\aa}rding polynomial. To make the definition of a Fang-Ma-G{\aa}rding polynomial parallel with the ideal ones, we use an equivalent definition by Theorem~1.1(3) of \cite{FangMaGarding}.

\begin{defn}[Theorem 1.1(3) of
\cite{FangMaGarding}, Definition 1.1 of \cite{FangMaIdealGarding}]
\label{FMGdefinition}
Let \(p\in\RR[x_1,\ldots,x_n]\) be nonzero.  For a
multi-index
\(\alpha=(\alpha_1,\ldots,\alpha_n)\in\mathbb Z_{\ge0}^n\), write
\[
  \partial^\alpha p
  =
  \frac{\partial^{|\alpha|}p}
       {\partial x_1^{\alpha_1}\cdots\partial x_n^{\alpha_n}},
  \qquad |\alpha|=\alpha_1+\cdots+\alpha_n.
\]
Then \(p\) is called a \emph{Fang-Ma-G{\aa}rding polynomial} if and only if,
for every \(\alpha\) such that \(\partial^\alpha p\not\equiv0\),
\begin{enumerate}
  \item the set \(\{\partial^\alpha p>0\}\) has a unique connected
  component \(\Upsilon_{\partial^\alpha p}\) passing PRT; and
  \item this component is contained in the corresponding components
  of all its first partial derivatives:
  \[
    \Upsilon_{\partial^\alpha p}
    \subseteq \Upsilon_{\partial_i\partial^\alpha p}
    \qquad(1\le i\le n).
  \]
\end{enumerate}
Here \(\Upsilon_h=\RR^n\) when \(h\equiv c>0\), and by convention \(\Upsilon_0=\RR^n\).
The component \(\Upsilon_p\) is called the \emph{Fang-Ma-G{\aa}rding component} of
\(p\), which is a generalization of the $\Upsilon$-stable component of \cite{LinInverseSigmaConvexity}. If, in addition, the Fang-Ma-G{\aa}rding component of $\partial^\alpha p$ is convex for all \(\alpha=(\alpha_1,\ldots,\alpha_n)\in\mathbb Z_{\ge0}^n\), we call $p$ an ideal Fang-Ma-G{\aa}rding polynomial, which is again a generalization of the convex property in \cite{LinInverseSigmaConvexity}. We denote the class of Fang-Ma-G{\aa}rding polynomials in \(n\) variables and the ideal ones by
\(\Ga_n\) and \(\IG_n\), and put \(\Ga=\bigcup_{n\ge1}\Ga_n, \IG=\bigcup_{n\ge1}\IG_n\).
\end{defn}

\begin{thm}[Theorem 8.13 of \cite{FangMaGarding}, Theorem 1.3 of \cite{FangMaIdealGarding}]
\label{kappa-polarization}
Let \(p\) be a polynomial of multidegree at most
\(\boldsymbol{\kappa}\), meaning
\(\deg_{x_i}p\le\kappa_i\) for every \(i\).  If \(p\in\Ga\) and its Fang-Ma-G{\aa}rding component \(\Upsilon_p\) is NRT, then
\(\Pi^\uparrow_{\boldsymbol{\kappa}}(p)\in\Ga\). If \(p\in\IG\), then
\(\Pi^\uparrow_{\boldsymbol{\kappa}}(p)\in\IG\).
\end{thm}

Fang and Ma have provided the following characterization of NRT polynomials:

\begin{lem}[Lemma 9.5 of \cite{FangMaGarding}]
\label{NRT-redundant}
Let $f\in \Ga_n$. Then the Fang--Ma--G{\aa}rding component $\Upsilon_f$ is NRT if and only if $\deg_{x_i} f\ge 1$ for every $i=1,2,\ldots,n$.
\end{lem}

We next recall the definition introduced by Lin.

\begin{defn}[Definition 2.1 of \cite{LinInverseSigmaConvexity}]
\label{right-Noetherian}
For a nonconstant univariate polynomial \(f\), let \(r(f)\) denote its
largest real root, provided such a root exists.  If
\(\deg f=d\), its \emph{root sequence} is
\[
  R(f)=\bigl(r(f),r(f'),\ldots,r(f^{(d-1)})\bigr),
\]
provided all entries exist.  The polynomial \(f\) is
\emph{right-Noetherian} if all these entries exist
and
\[
  r(f)\ge r(f')\ge\cdots\ge r(f^{(d-1)}).
\]

The polynomial \(f\) with degree $d\ge 2$ is
\emph{strictly right-Noetherian} if all these entries exist
and
\[
  r(f)> r(f')\ge\cdots\ge r(f^{(d-1)}).
\]

The polynomial \(f\) with degree $d\ge 2$ is
\emph{strongly strictly right-Noetherian} if all these entries exist
and
\[
  r(f)> r(f') >\cdots > r(f^{(d-1)}).
\]

\end{defn}

The following is immediate from Section~9.2 of \cite{FangMaGarding} and the fact that any interval is convex.
\begin{lem}[Section 9.2 of \cite{FangMaGarding}]
\label{FangMaLinCorrespondence}
A monic non-constant univariate polynomial $p(x)$ belongs to $\Ga_1=\IG_1$ if and only if it is right-Noetherian.
\end{lem}

Next, let us recall the definition of proper position from \cite{FangMaGarding} and ideal proper position from \cite{FangMaIdealGarding}.

\begin{defn}[Definition 10.9 of \cite{FangMaGarding}, Definition 9.4 of \cite{FangMaIdealGarding}]
Suppose $f_2(x), f_1(x) \in \Ga[x] \setminus \{0\}$. The pair $(f_2,f_1)$ is said to be \emph{in proper position}, written as $f_2 \prec f_1$, if
 \[f_2(x)y +f_1(x) \in \Ga[x,y].\]
 Suppose $f_2(x), f_1(x) \in \IG[x] \setminus \{0\}$. The pair $(f_2,f_1)$ is said to be \emph{in ideal proper position}, written as $f_2\idealprec f_1$, if
 \[f_2(x)y +f_1(x) \in \IG[x,y].\]
\end{defn}

Using this definition and Theorem \ref{kappa-polarization}, it is almost trivial to prove the following corollary. However, since it is not explicitly stated in either \cite{FangMaGarding} or \cite{FangMaIdealGarding}, we write it out here for completeness:

        \begin{cor}
        \label{polarizationproperposition}
        Suppose that $f_2, f_1$ are univariate monic polynomials of degree $1\le d\le n$. If $f_2, f_1\in \Ga_1$ and $f_2\prec f_1$, then $\Pol_n f_2 \prec \Pol_n f_1$. If $f_2, f_1\in \IG_1$ and $f_2\idealprec f_1$, then $\Pol_n f_2 \idealprec \Pol_n f_1$.
        \end{cor}
        \begin{proof}
        If $f_2\prec f_1$, by definition, $f_2(x)y+f_1(x)\in \Ga[x,y]$. By Lemma \ref{NRT-redundant}, its Fang-Ma-G{\aa}rding component is NRT. By Theorem \ref{kappa-polarization}, if we choose $\kappa=(n,1)$, then the $\kappa$-polarization 
        \[
        g_2(\mathbf{x})y+g_1(\mathbf{x})\in \Ga[\mathbf{x},y],
        \]
        where $g_2=\Pol_n f_2$ and $g_1=\Pol_n f_1$. Thus, $g_2\prec g_1$ by definition. The case $f_2\idealprec f_1$ can be proved similarly.
        \end{proof}
        
    Recall Theorem 6.14 of \cite{FangMaGarding}.
    
    \begin{thm}[Theorem 6.14 of \cite{FangMaGarding}]
\label{PositiveDerivative}
Let $g_2$ and $g_1$ be two nontrivial multi-affine Fang-Ma-G{\aa}rding polynomials. Then the following are equivalent:
\begin{enumerate}
    \item $g_2\prec g_1$;
    \item $\Upsilon_{g_2} \cap\{g_1>0\}=\Upsilon_{g_1}$;
    \item For every $i=1,2,\ldots,n$, $g_1\partial_i g_2-g_2\partial_i g_1 \le 0$ on $\Upsilon_{g_2}$.
\end{enumerate}
\end{thm}

Let us also recall the definition of the dominance relationship by \cite{FangMaGarding}, which extends Lin's $\Upsilon$-dominance relationship in \cite{LinInverseSigmaConvexity}.

\begin{defn}[Definition 10.8 of \cite{FangMaGarding}]\label{Dominance}
Let \(g_1, g_2 \in \Ga\). If
  \[
\Upsilon_{\partial^\alpha g_1} \subset \{\partial^\alpha g_2 > 0\} \quad \text{for every nontrivial partial derivative } \partial^\alpha g_2,
\]
then we say that \(g_1\) \textit{dominates} \(g_2\), and write \(g_2 \triangleleft g_1\). This is also equivalent to $\Upsilon_{\partial^\alpha g_1}\subset\Upsilon_{\partial^\alpha g_2}$. Note that \(g_2 \triangleleft g_1\) implies \(\deg g_1 \ge \deg g_2\).
\end{defn}

Then we have the following results:
        
\begin{thm}[Theorem 10.10 of \cite{FangMaGarding}]
\label{DominancePreserve}
Suppose that $f_2, f_1$ are univariate monic polynomials of degree $1\le d_2\le d_1\le n$. If $f_2, f_1\in \Ga_1$ and $f_2\triangleleft f_1$, then $\Pol_n f_2 \triangleleft \Pol_n f_1$.
\end{thm}

\begin{prop}
[Corollary 10.11 of \cite{FangMaGarding}]\label{Dominance10.11}
Let \(g_2, g_1 \in \Ga_n\) be such that \(g_2 \triangleleft g_1\). Then for every $i=1,2,\ldots,n$, 
\[
\frac{g_1\partial_i g_2-g_2\partial_i g_1}{g_2^2}=-\partial_i(\frac{g_1}{g_2})\le 0
\] on \(\Upsilon_{g_1}\subset \Upsilon_{g_2}\).
\end{prop}

For univariate right-Noetherian polynomials, proper position can be characterized by Theorem 10.12 of \cite{FangMaGarding}.

\begin{thm}[Theorem 10.12 of \cite{FangMaGarding}]
\label{PositiveUnivariable}
Let $f_2, f_1\in \Ga_1$ be univariate monic polynomials of the same degree $d$. Then $f_2\prec f_1$ if and only if, for every $k=0,1,\ldots,d-1$, $r(f_2^{(k)}) \le r(f_1^{(k)})$ and $f_1^{(k)}(x)\le 0$ for all $x \in [r(f_2^{(k)}),r(f_1^{(k)}))$.
\end{thm}

We will study the strict version of Theorem \ref{PositiveDerivative} in the next section. Let's also recall the ideal version.

\begin{thm}[Remark 9.5 of \cite{FangMaIdealGarding}]
\label{ConcaveQuotient}
For non-trivial $g_1, g_2\in\IG$, $g_2\idealprec g_1$ if and only if $g_2\prec g_1$ and $\frac{\partial^\alpha g_1(x)}{\partial^\alpha g_2(x)}$ is a concave function in $\Upsilon_{\partial^\alpha g_2(x)}$ for all \(\alpha=(\alpha_1,\ldots,\alpha_n)\in\mathbb Z_{\ge0}^n\) such that $\partial^\alpha g_2(x) \not\equiv 0$.
\end{thm}

As a corollary, for univariate right-Noetherian polynomials, ideal proper position can be characterized in the following way:

\begin{cor}
\label{IdealProperUnivariable}
Let $f_2, f_1\in \Ga_1=\IG_1$ be univariate monic polynomials of the same degree $d$. Then $f_2\idealprec f_1$ if and only if $f_2\prec f_1$ and,  for every $k=0,1,\ldots,d-1$, $\frac{f_1^{(k)}(x)}{f_2^{(k)}(x)}$ is a concave function in $(r(f_2^{(k)}),\infty)$.
\end{cor}

By definition, the dominance condition can also be characterized in the following way:

\begin{prop}
\label{DominanceUnivariable}
Let $f_2, f_1\in \Ga_1$ be univariate monic polynomials of degree $1\le d_2\le d_1\le n$. Then $f_2\triangleleft f_1$ if and only if, for every $k=0,1,\ldots,d_2-1$, $r(f_2^{(k)}) \le r(f_1^{(k)})$.
\end{prop}

Let us recall the following standard results. They are known to experts and can be proved using the Courant-Fischer-Weyl min-max principle, the Ky Fan-Lidskii inequality, and Rado's corollary of the Hardy-Littlewood-P\'{o}lya theorem. For example, see Corollary 2.20 and Proposition 2.21 of \cite{ChenNieXuNumericalCriterion} for more details.

\begin{prop}
\label{PRTmatrix}
Let $\Upsilon\subseteq\mathbb{R}^n$ be invariant under coordinate
permutations and satisfy PRT. If $A, B, C, D\in\Herm(n)$, $C$ is positive definite, and $A-C$, $B$, and $D-B$ are positive semi-definite, then $\lambda(A^{-1}B)\in\Upsilon$ implies that \[
  \lambda(C^{-1}D)\in\Upsilon.
\]
\end{prop}

\begin{prop}
\label{convexmatrix}
Let $\Upsilon\subseteq\mathbb{R}^n$ be convex and invariant under
coordinate permutations.  Then
\[
  \{A\in\Herm(n):\lambda(A)\in\Upsilon\}
\]
is convex.
\end{prop}

\section{Strict proper position}

In this section, we give a characterization of univariate right-Noetherian polynomials such that their polarizations satisfy the strict inequality of Theorem \ref{PositiveDerivative}(3).

\begin{prop}\label{Strict proper position}
Assume that $f_2,f_1$ are univariate monic right-Noetherian polynomials of the same degree $d$ such that $f_2\prec f_1$. Set $r_2=r(f_2)$, let $m$ be the multiplicity of $r_2$ as a root of $f_2$, and put $F=g_1/g_2$, where $g_i=\Pol_n f_i$, $i=1,2$ for $n\ge d$.
Then the following are equivalent:
\begin{enumerate}
\item[(i)] $\partial_{\lambda_i}F(\lambda)>0$ for every
      $\lambda\in\Upsilon_{g_2}$ and every $1\le i\le n$.
\item[(ii)] $\displaystyle\lim_{x\downarrow r_2}f_1(x)/f_2(x)=-\infty$.
\item[(iii)] There is an integer $0\le q<m$ such that
      $f_1^{(k)}(r_2)=0$ for $0\le k<q$ and
      $f_1^{(q)}(r_2)<0$.
\item[(iv)] There is an integer $0\le q<m$ such that
      $f_1^{(k)}(r_2)=0$ for $0\le k<q$ and
      $f_1^{(q)}(r_2)\not=0$.
\end{enumerate}
We say that $f_2$ is in strictly proper position relative to $f_1$, or $f_2\prec f_1$ strictly, when these equivalent conditions hold.
\end{prop}

\begin{proof}
A common translation of the univariate variable and of all the polarization variables allows us to assume that $r_2=0$. Now we prove that (ii), (iii) and (iv) are equivalent to each other.
Since $m$ is the multiplicity of $r_2$ as a root of $f_2$, we have $f_2^{(k)}(0)=0, k<m$, and $f_2^{(m)}(0)>0$. So, we have
\[
f_1(x) = \sum_{k=0}^{d} \frac{f_1^{(k)}(0)}{k!} x^k ,\quad f_2(x) = \sum_{k=0}^{d} \frac{f_2^{(k)}(0)}{k!} x^k= \sum_{k=m}^{d} \frac{f_2^{(k)}(0)}{k!} x^k .
\]
Suppose the first nonzero term of $f_1^{(k)}(0), 0\le k \le d$ is $f_1^{(q)}(0).$
When $x\downarrow 0$, $f_1\sim\frac{f_1^{(q)}(0)}{q!} x^q$ while $f_2\sim\frac{f_2^{(m)}(0)}{m!} x^m>0$, where $f\sim g$ means that $\lim_{x\downarrow 0}\frac{f}{g}=1$. Hence, (ii) and (iii) are equivalent. It is obvious that (iii)$\Rightarrow$(iv). For (iv)$\Rightarrow$(iii), we calculate 
\[f_1'f_2-f_1f_2'\sim \frac{q-m}{q!m!}f_1^{(q)}(0)f_2^{(m)}(0) x^{m+q-1}.\] Since $\frac{q-m}{q!m!}<0$, we see that $f_1^{(q)}(0)>0$ will contradict $(f_1/f_2)'\ge 0$ on $(0,\infty)$. So we have (iii).

We first prove (i)$\Rightarrow$(iv). Suppose that (iv) is false, and we need to show that (i) is false. Then $f_1^{(k)}(0)=0$ for $k<m$, and both polarizations have the
expansions
\begin{equation}\label{eq:strict-proper-taylor}
 g_j(\lambda)=\sum_{k=m}^d
 \frac{f_j^{(k)}(0)}{k!\binom nk}\sigma_k(\lambda),
 \qquad j=1,2.
\end{equation}
Choose $\lambda=(\lambda_1,\ldots,\lambda_m,0,\ldots,0)$ with all $\lambda_i>0$. Since there are only $m$ non-zero terms in $\lambda$, $\sigma_{k}(\lambda)=0$ for all $k>m$. By the right-Noetherian condition, $f_2^{(m)}(0)>0$. Thus $g_2(\lambda)>0$. Moreover, for all $t>r_2=0$, $t\mathbf{1}\in \Upsilon_{g_2}$. Then $\lambda+t\mathbf{1}\in\Upsilon_{g_2}$ since $\Upsilon_{g_2}$ passes PRT.
So $\lambda$ is in the Fang-Ma-G{\aa}rding connected component $\Upsilon_{g_2}$ of $\{g_2>0\}$.
Note that only the $k=m$ term survives
in \eqref{eq:strict-proper-taylor} because $\sigma_{k}(\lambda)=0$ for all $k>m$. So $F(\lambda)=f_1^{(m)}(0)/f_2^{(m)}(0)$ is independent of each $\lambda_i$. By taking the derivative in the $\lambda_i$ direction for $i=1,\ldots,m$, we see that (i) is not true.

We next prove (iii)$\Rightarrow$(i). Let $0< \varepsilon <- \frac{f_1^{(q)}(0)}{q!\binom{n}{q}}$.
We claim that
\begin{equation}\label{eq:strict-proper-perturbation}
 \widetilde f_1=f_1+\varepsilon \binom{n}{q} x^q
 \quad\text{is right-Noetherian and satisfies}\quad
 f_2\prec\widetilde f_1.
\end{equation}
Note that for $k>q$, the derivative $\tilde f_1^{(k)}=f_1^{(k)}$ is unchanged. Moreover, for $k\le q<m$, first note that $f_2^{(k)}$ has a root at $0$, so $r(f_2^{(k)})\ge r(f_2)=0$. This implies that $r(f_2^{(k)})=0$ by the right-Noetherian condition. By Theorem \ref{PositiveUnivariable}, it suffices to show that for all $k=0,1,...,q\le d-1$, there is a unique positive root $R_k$ of $\tilde f_1^{(k)}$ such that $R_k\ge R_{k+1}=r(\tilde f_1^{(k+1)})$ unless $k=d-1$, and $\tilde f_1^{(k)}(x)<0$ for all $x \in (0,R_k)$.

For $k=q$,
$\tilde f_1^{(q)}(0)=f_1^{(q)}(0)+\varepsilon \binom{n}{q}q!<0$. If $f_1^{(q)}$ has degree at least two, and $r(f_1^{(q+1)})>0$, then by Theorem \ref{PositiveUnivariable} applied to $f_2\prec f_1$, we see that
$f_1^{(q+1)}\le0$ on $[0,r(f_1^{(q+1)})]\subset[r(f_2^{(q+1)}),r(f_1^{(q+1)})]$.
Also, $f_1^{(q+1)}>0$ on $(r(f_1^{(q+1)}),\infty)$. Therefore $\tilde f_1^{(q)}=f_1^{(q)}+\varepsilon \binom{n}{q}q!$ has a unique positive root $R_q>r(f_1^{(q+1)})=R_{q+1}$ and is negative on
$[0,R_q)$. If $f_1^{(q)}$ has degree at least two, and $r(f_1^{(q+1)})\le 0$, then $\tilde f_1^{(q)}$ is increasing on $[0,\infty)$. Since $\tilde f_1^{(q)}(0)<0$, it has a unique positive root $R_{q}>0\ge R_{q+1}$, and $f_1^{(q)}<0$ on $[0,R_q)$. If $f_1^{(q)}$ is linear, $f_1^{(q+1)}\equiv const>0$ on $(-\infty,+\infty)$. The same conclusion holds. This verifies the required condition for $k=q$.

For $k<q$, by induction, we assume that the condition holds for $k+1$. Then $R_{k+1}>0$, and $\tilde f_1^{(k+1)}$ is negative on $(0,R_{k+1})$ and positive on $(R_{k+1},\infty)$. Since $\tilde f_1^{(k)}(0)=0$, we see that $\tilde f_1^{(k)}$ is
negative on $(0,R_{k+1}]$, strictly increasing thereafter,
and has a unique positive root $R_k>R_{k+1}$. It is therefore negative on $(0,R_k)$. This proves
\eqref{eq:strict-proper-perturbation}.

Next, using $r(f_2^{(k)})=0$ for $k<m$, we see that $x^m\vartriangleleft f_2$ by Proposition \ref{DominanceUnivariable}. Then by Theorem \ref{DominancePreserve}, $\sigma_m\vartriangleleft g_2$. So Proposition \ref{Dominance10.11} implies that \[
\frac{g_2\partial_i \sigma_m-\sigma_m\partial_i g_2}{g_2^2}=\partial_i(\frac{\sigma_m}{g_2})\le 0
\]
on $\Upsilon_{g_2}\subset\Upsilon_{\sigma_m}=\Gamma_m=\{\sigma_1>0,\ldots,\sigma_m>0\}$.
Recall that the classical Newton inequalities give
$\partial_i(\sigma_q/\sigma_m)<0$ on $\Gamma_m$ whenever $q<m$. For example, see Theorem 2.16 of \cite{SpruckGeometricAspects}. Now $\sigma_q/g_2=(\sigma_q/\sigma_m)(\sigma_m/g_2)$ is a product of two positive functions, whose $\partial_i$ derivatives are respectively negative and nonpositive. This proves
\[
 \partial_i\left(\frac{\sigma_q}{g_2}\right)<0
 \quad\text{on }\Upsilon_{g_2},\qquad i=1,\ldots,n.
\]

Finally, by \eqref{eq:strict-proper-perturbation} and Corollary \ref{polarizationproperposition}, $g_2\prec \operatorname{Pol}_n\widetilde f_1=g_1+\varepsilon \sigma_q$. Then Theorem \ref{PositiveDerivative}
gives
$\partial_i(g_1/g_2)\ge-\varepsilon\partial_i(\sigma_q/g_2)>0$
on $\Upsilon_{g_2}$. This proves (i).
\end{proof}

\begin{remark}
\label{strict-prec-1d}
If $f_2\prec f_1$ strictly, then $(\frac{f_1}{f_2})'>0$ on $(r(f_2),\infty)$. However, the counterexample $f_2(x)=x(x+1)$ and $f_1(x)=x(x-1)$ shows that the converse fails.
\end{remark}

We also need the following stronger version of strict proper position:
\begin{defn}
\label{SSPrec}
Assume that $f_2,f_1$ are univariate monic right-Noetherian polynomials of the same degree $d$. Then we say that $f_2\prec f_1$ strongly strictly if $f_2^{(k)}\prec f_1^{(k)}$ strictly for all $k=0,1,\ldots,d-1$.
\end{defn}

This definition is useful to prove the following:
\begin{lem}
\label{Linear-SSPrec}
Assume that $f_2,f_1$ are univariate monic right-Noetherian polynomials of the same degree $d$, and $f_2\prec f_1$ strongly strictly. Then for any $C>-1$, the linear combination $f_C=f_1+Cf_2$ is strongly strictly right-Noetherian, and $f_2\prec f_C$ strongly strictly.
\end{lem}
\begin{proof}
For any $k=0,\ldots,d-1$, the function $(\frac{f_1^{(k)}}{f_2^{(k)}})'>0$ on $(r(f_2^{(k)}),\infty)$. So $\frac{f_1^{(k)}}{f_2^{(k)}}$ strictly increases from $-\infty$ to $1$. It passes $0$ at $r(f_1^{(k)})$ and passes $-C$ at the largest real root $r(f_C^{(k)})$. In particular, $r(f_2^{(k)})<r(f_C^{(k)})$. Now we take the derivative of $f_C^{(k)}=(\frac{f_C^{(k)}}{f_2^{(k)}})f_2^{(k)}$, and get
\[f_C^{(k+1)}=(\frac{f_C^{(k)}}{f_2^{(k)}})'f_2^{(k)}+(\frac{f_C^{(k)}}{f_2^{(k)}})f_2^{(k+1)}.\]
The term $(\frac{f_C^{(k)}}{f_2^{(k)}})'=(\frac{f_1^{(k)}}{f_2^{(k)}})'>0$ on $(r(f_2^{(k)}),\infty)$. So $f_C^{(k+1)}>0$ on $[r(f_C^{(k)}),\infty)$. This implies that $r(f_C^{(k+1)}) < r(f_C^{(k)})$ for all $k=0,\ldots,d-2$. So $f_C$ is strongly strictly right-Noetherian. Moreover, the function $\frac{f_C^{(k)}}{f_2^{(k)}}=\frac{f_1^{(k)}}{f_2^{(k)}}+C$ strictly increases from $-\infty$ to $1+C$. So $f_2\prec f_C$ strongly strictly.
\end{proof}
We will use a similar argument for a different perturbation later.

\section{Volume comparison}
In Chu and Lee's paper \cite{ChuLeeHypercriticalDHYM}, they used the maximum principle to obtain the bound on $\frac{g_1(\chi_\varphi)}{g_2(\chi_\varphi)}$ along the flow. This implies that the geometric volume form $g_2(\chi_\varphi)$ is comparable to the traditional volume form $\chi_{\varphi}^n$ in their case by a quite trivial calculation. This section proves the same result under the strongly strict proper position condition.

\begin{prop}\label{prop:comparison}
Define
\[
 f_i(x)=\sum_{k=0}^n c_k^{(i)}\binom nk x^k,
 \qquad g_i(\lambda)=\sum_{k=0}^n c_k^{(i)}\sigma_k(\lambda),
 \qquad c_n^{(i)}=1,
\]
and
\[
 g_i(\chi_\varphi)=\sum_{k=0}^n c_k^{(i)}\binom nk
       \chi_\varphi^k\wedge\omega^{n-k},
 \qquad \chi_\varphi=\chi+\ddc\varphi,\quad i=1,2.
\]
Assume that $f_2\prec f_1$ strongly strictly as in Definition \ref{SSPrec}, and $c_{n-1}^{(2)}=-1$. Then for every $C>0$, there exists $\kappa_C>0$ such that
\[
 \kappa_C\chi_\varphi^n\le g_2(\chi_\varphi)
       \le\kappa_C^{-1}\chi_\varphi^n
\]
for every $\varphi\in\HH$ satisfying
$\frac{g_1(\chi_\varphi)}{g_2(\chi_\varphi)}\ge -C$.
\end{prop}
\begin{proof}
We prove a pointwise polynomial estimate. Write
$F=g_1/g_2$ and $\mathbf{1}=(1,\ldots,1)$. For any $C>0$, by Lemma \ref{Linear-SSPrec}, $f_{C+1}=f_1+(C+1)f_2$ is strongly strictly right-Noetherian and $f_2\prec f_{C+1}$ strongly strictly. Let $g_{C+1}$ be the polarization of $f_{C+1}$; then $g_2\prec g_{C+1}$. So
\[\{F \ge -C\}\cap\Upsilon_{g_2}\subset \{F > -C-1\}\cap\Upsilon_{g_2}=\Upsilon_{g_{C+1}}.\]

For $0\le k<n$, since $f_2\prec f_{C+1}$ strongly strictly,
$r(f_{C+1}^{(k)})>r(f_2^{(k)})$. Choose
\[
\delta=\frac{1}{2}\min_{0\le k<n}(r(f_{C+1}^{(k)})-r(f_2^{(k)}))>0.
\]
The polynomial $f_2(x-\delta)$ is monic and right-Noetherian, and its $k$-th derivative has largest root $r(f_2^{(k)})+\delta<r(f_{C+1}^{(k)})$.
By Proposition \ref{DominanceUnivariable}, $f_2(x-\delta) \triangleleft f_{C+1}(x)$. So $g_2(\lambda-\delta\mathbf{1}) \triangleleft g_{C+1}(\lambda)$ by Theorem \ref{DominancePreserve}. By Definition \ref{Dominance}, we have \[
 \{\lambda\in\Upsilon_{g_2}:F(\lambda)\ge-C\}
 \subset\Upsilon_{g_{C+1}} \subset \Upsilon_{g_2(\lambda-\delta\mathbf{1})}
 =\delta\mathbf{1}+\Upsilon_{g_2}.
\]

By our assumption, $r(f_2)\ge r(f_2^{(n-1)})=1$. Fix $A>r(f_2)$. For any $\mu\in\Upsilon_{g_2}$, let $\lambda=\mu+\delta\mathbf{1}$.
Consider the polynomial
\[q_\mu(t):=g_2(\mu+t\mathbf{1})=\sum_{k=0}^{n} q_k t^k.\]
Then by Definition \ref{FMGdefinition}(2), for all $0\le k \le n$, \[q_k=\frac{1}{k!}\frac{d^k}{dt^k}|_{t=0}g_2(\mu+t\mathbf{1}) > 0.\] 
So
\begin{equation}\label{eq:comparison-translation}
 g_2(\lambda+A\mathbf{1}) = \sum_{k=0}^{n} q_k (\delta+A)^k \le \left(\frac{\delta+A}{\delta}\right)^n\sum_{k=0}^{n} q_k \delta^k
 =\left(\frac{\delta+A}{\delta}\right)^n g_2(\lambda).
\end{equation}
On the other hand, \begin{equation}\label{eq:comparison-positive-coefficients}
 g_2(\lambda+A\mathbf{1})=\Pol_n f_2(x+A)=\Pol_n \sum_{k=0}^n\frac{f_2^{(k)}(A)}{k!}x^k=\sum_{k=0}^n a_k\sigma_k(\lambda),
 \quad a_k=\frac{f_2^{(k)}(A)}{k!\binom nk}>0.
\end{equation}
The positivity follows from right-Noetherianity and
$A>r(f_2)\ge r(f_2^{(k)})$ for $k<n$, while $a_n=1$.
Since $c_n^{(i)}=1$ and $c_{n-1}^{(2)}=-1$, by Definition \ref{FMGdefinition}(2), 
\begin{equation}
\label{eq:denominator-cone-normalization}
\Upsilon_{g_2}\subset\bigcap_{|\alpha|=n-1}\Upsilon_{\partial^\alpha g_2}=\bigcap_{i=1}^{n}\{\lambda_i>1\}.
\end{equation}
Since $a_n=1$ and all the other coefficients
in \eqref{eq:comparison-positive-coefficients} are positive,
$g_2(\lambda+A\mathbf{1})\ge\prod_{i=1}^n\lambda_i$.
Combining this with \eqref{eq:comparison-translation}, we obtain
\[
 \left(\frac{\delta}{\delta+A}\right)^n
       \prod_{i=1}^n\lambda_i
 \le g_2(\lambda) = \sum_{k=0}^n c_k^{(2)}\sigma_{k}(\lambda)
 \le \sum_{k=0}^n\binom nk|c_k^{(2)}|\prod_{i=1}^n\lambda_i
\]
for all $\lambda\in\Upsilon_{g_2}$ such that $\frac{g_1(\lambda)}{g_2(\lambda)}\ge -C$. We have used $\sigma_k(\lambda)\le \binom{n}{k}\prod_{i=1}^n\lambda_i$ since $\lambda_i>1$. This proves the required bound.
\end{proof}

\section{A compatible perturbation of the polynomial pair}
\label{sec:polynomial-perturbation}
We verify the hypotheses needed to apply the preceding algebraic
results and the subsequent analytic arguments to a perturbed pair
$(f_3,f_4)$. Throughout this section, assume $V_1=0$, as in
Theorem~\ref{thm:main}, and write $\HH_2=\HH$ for the original
admissible space. Choose $\epsilon_4>0$ sufficiently small that
$1+\epsilon_4 F_\omega(\chi)>0$ everywhere on $X$, and then choose
$\epsilon_3>0$ sufficiently small depending on this fixed $\epsilon_4$.
Define
\begin{equation}\label{eq:perturbed-polynomial-pair}
 f_3(x)=f_1(x)-\epsilon_3 f_2(x)+\epsilon_3 v_2,
 \qquad
 f_4(x)=f_2(x)+\epsilon_4 f_1(x),
 \qquad
 v_2=\frac{V_2}{\int_X\omega^n}>0.
\end{equation}
We always take $\epsilon_3<1$, so both polynomials are positive multiples of monic polynomials of degree $n$. Put $g_i=\operatorname{Pol}_n f_i$ for $i=3,4$, and
define $\HH_4$ by replacing $g_2$ with $g_4$ in the definition of
$\HH$. The order of the two small parameters is essential:
$\epsilon_4$ is fixed before $\epsilon_3$ is chosen.

\begin{lem}\label{lem:f4-position}
The polynomial $f_4$ is strongly strictly right-Noetherian. Moreover,
$f_2\idealprec f_4\idealprec f_1$, and $f_2\prec f_4\prec f_1$ strongly strictly.
The admissible spaces satisfy
\begin{equation}\label{eq:H4-subset-H2}
 \HH_4=\{\varphi\in\HH_2:
               1+\epsilon_4 F_\omega(\chi_\varphi)>0\}
 \subseteq\HH_2,
 \qquad 0\in\HH_4.
\end{equation}
\end{lem}
\begin{proof}
For $0\le k<n$, by Proposition~\ref{Strict proper position}, Remark \ref{strict-prec-1d} and Corollary~\ref{IdealProperUnivariable}, $\frac{f_1^{(k)}}{f_2^{(k)}}$ strictly increases from $-\infty$ at $r(f_2^{(k)})$ to $1$ at $\infty$ and is concave on $(r(f_2^{(k)}),\infty)$. So it passes $-\frac{1}{\epsilon_4}$ at $r(f_4^{(k)})$ and passes $0$ at $r(f_1^{(k)})$. So \begin{equation}\label{eq:f4-root-interlacing}
 r(f_2^{(k)})<r(f_4^{(k)})<r(f_1^{(k)}).
\end{equation}
By Lemma \ref{Linear-SSPrec}, $f_4=\epsilon_4(f_1(x)+\epsilon_4^{-1}f_2)$ is strongly strictly right-Noetherian and $f_2\prec f_4$ strongly strictly.
By Corollary~\ref{IdealProperUnivariable}, $f_2\idealprec f_4$.

By Theorem \ref{PositiveUnivariable}, $f_4\prec f_1$ by \eqref{eq:f4-root-interlacing} and $f_1^{(k)}\le 0$ on $[r(f_4^{(k)}),r(f_1^{(k)}))\subset [r(f_2^{(k)}),r(f_1^{(k)}))$. Since $r(f_4^{(k)})$ is a simple root of $f_4^{(k)}$ but is not a root of $f_1^{(k)}$, $f_4\prec f_1$ strongly strictly by Proposition \ref{Strict proper position}.
On $(r(f_4^{(k)}),\infty)$, we know that $1+\epsilon_4 (f_1^{(k)}/f_2^{(k)})=\frac{f_4^{(k)}}{f_2^{(k)}}>0$. By taking the derivative of
$\frac{f_1^{(k)}}{f_4^{(k)}}=\frac{f_1^{(k)}/f_2^{(k)}}{1+\epsilon_4 (f_1^{(k)}/f_2^{(k)})}$, we get
\begin{equation}\label{eq:f4-strict-quotient-increasing}
(\frac{f_1^{(k)}}{f_4^{(k)}})'=
       \frac{(f_1^{(k)}/f_2^{(k)})'}{(1+\epsilon_4 (f_1^{(k)}/f_2^{(k)}))^2}>0
       \end{equation}
and
\begin{equation}\label{eq:f4-strict-quotient-concavity}
 (\frac{f_1^{(k)}}{f_4^{(k)}})''=
       \frac{(f_1^{(k)}/f_2^{(k)})''}{(1+\epsilon_4 (f_1^{(k)}/f_2^{(k)}))^2}
       -\frac{2\epsilon_4((f_1^{(k)}/f_2^{(k)})')^2}{(1+\epsilon_4 (f_1^{(k)}/f_2^{(k)}))^3}<0.
\end{equation}
So $f_4\idealprec f_1$. The remaining part has been proved by Lemma \ref{Linear-SSPrec}.
\end{proof}

\begin{prop}\label{prop:f3-f4-hypotheses}
For every fixed $\epsilon_4$ as above, there exists
$\epsilon_{3,0}(\epsilon_4)>0$ such that, whenever
$0<\epsilon_3<\epsilon_{3,0}(\epsilon_4)$, $f_3$ is strongly strictly right-Noetherian, $f_4\idealprec f_3$, and $f_4\prec f_3$ strongly strictly.
\end{prop}
\begin{proof}
For $1\le k<n$,
\[\frac{f_3^{(k)}}{f_4^{(k)}}=\frac{f_1^{(k)}-\epsilon_3 f_2^{(k)}}{f_4^{(k)}}=\frac{f_1^{(k)}-\epsilon_3 f_4^{(k)}+\epsilon_3\epsilon_4f_1^{(k)}}{f_4^{(k)}}=(1+\epsilon_4\epsilon_3)\frac{f_1^{(k)}}{f_4^{(k)}}-\epsilon_3.\]
On $(r(f_4^{(k)}),\infty)$, it is strictly concave by
\eqref{eq:f4-strict-quotient-concavity}, and its left
endpoint limit is $-\infty$. For $k=0$ the extra
term $\epsilon_3 v_2$ in \eqref{eq:perturbed-polynomial-pair} gives
\begin{equation}\label{eq:zeroth-perturbed-quotient}
 \frac{f_3}{f_4}
 =(1+\epsilon_4\epsilon_3)\frac{f_1}{f_4}
  -\epsilon_3
  +\frac{\epsilon_3 v_2}{f_4}.
\end{equation}
We must check that the last term preserves concavity on $(r(f_4),\infty)$.

Let $r_4=r(f_4)$. Lemma~\ref{lem:f4-position} gives
$f_4'(r_4)>0$, $f_4''(r_4)>0$, and $f_1(r_4)<0$. Since
\begin{equation*}
(\frac{1}{f_4})''=(\frac{-f_4'}{f_4^2})'=
       \frac{-f_4''f_4^2+2f_4f_4'^2}{f_4^4},
 \end{equation*}      
 \begin{equation*}
 (\frac{f_1}{f_4})''=(\frac{f_1'f_4-f_4'f_1}{f_4^2})'=
       \frac{(f_1''f_4-f_4''f_1)f_4^2-2f_4f_4'(f_1'f_4-f_4'f_1)}{f_4^4},
\end{equation*}
we have \[
(\frac{1}{f_4})''\sim \frac{2}{f_4'(r_4)}\frac{1}{(x-r_4)^3}, \quad (\frac{f_1}{f_4})''\sim \frac{2f_1(r_4)}{f_4'(r_4)}\frac{1}{(x-r_4)^3}
\]
near $r_4$. So the ratio $\frac{|(1/f_4)''|}{|-(f_1/f_4)''|}$ is bounded near $r_4$.
Recall that
$f_i(x)=\sum_{k=0}^{n}c_{k}^{(i)}\binom{n}{k}x^k$ and $c_{n}^{(i)}=1$ for $i=1,2$. Strict proper position for the $(n-1)$st derivative pair gives $c_{n-1}^{(1)}<c_{n-1}^{(2)}$. At infinity,
 \begin{equation*}
\frac{f_1(x)}{f_4(x)}=\frac{1}{1+\epsilon_4}-\frac{n(c_{n-1}^{(2)}-c_{n-1}^{(1)})}{(1+\epsilon_4)^2x}+O(x^{-2}), 
\end{equation*}
so $-(\frac{f_1(x)}{f_4(x)})''$ has a positive leading term of order $x^{-3}$, whereas $(\frac{1}{f_4(x)})''=O(x^{-n-2})$.
So the ratio $\frac{|(1/f_4)''|}{|-(f_1/f_4)''|}$ also has a bound at infinity. By compactness, the ratio $\frac{|(1/f_4)''|}{|-(f_1/f_4)''|}$ has a bound $K_{\epsilon_4}$ on $(r_4,\infty)$. Choose $\epsilon_3$ small enough that
$\epsilon_3 v_2 K_{\epsilon_4}<1/2$. Then $(\frac{f_3}{f_4})''<0$ on $(r_4,\infty)$.
Shrink $\epsilon_3$ further so that
\begin{equation*}(1+\epsilon_4\epsilon_3)f_1(r_4)+\epsilon_3 v_2<0.
\end{equation*}
This is possible because $f_1(r_4)<0$. Then $f_3(r_4)<0$ by \eqref{eq:zeroth-perturbed-quotient} and therefore $\lim_{x\downarrow r_4}\frac{f_3}{f_4}=-\infty$.

For each $0\le k<n$, we have now proved that
$\frac{f_3^{(k)}}{f_4^{(k)}}$ is strictly concave on $(r(f_4^{(k)}),\infty)$, has a left endpoint limit $-\infty$, and tends to $\frac{1-\epsilon_3}{1+\epsilon_4}$ at infinity. By a similar argument, it is also strictly increasing. Then a similar argument to that in Lemma \ref{Linear-SSPrec} finishes the proof.
\end{proof}

Recall that for $i=1,2$,
\[
 f_i(x)=\sum_{k=0}^n c_k^{(i)}\binom nk x^k,
 \qquad g_i(\lambda)=\sum_{k=0}^n c_k^{(i)}\sigma_k(\lambda),
 \qquad c_n^{(i)}=1, \quad c_{n-1}^{(2)}=-1.
\]

\begin{lem}\label{lem:f3-strict-cone-inclusion}
Suppose $n\ge2$ and the parameters satisfy
Proposition~\ref{prop:f3-f4-hypotheses}. There is
$0<\delta_{\epsilon_3}<1$ such that
\begin{equation}\label{eq:f3-strict-cone-inclusion}
 (1-\delta_{\epsilon_3})\overline{\Upsilon_{g_3}^{\,1}}
 \subset\delta_{\epsilon_3}\mathbf1+
                  \Upsilon_{g_1}^{\,1}.
\end{equation}
Here $\Upsilon_{g_i}^{\,1}:=
\bigcap_{j=1}^n\Upsilon_{\partial_j g_i}$ for $i=1,3$.
\end{lem}
\begin{proof}
For $1\le k<n$, the strongly strict proper-position assumptions give
$r(f_1^{(k)})>r(f_2^{(k)})$, so $f_2^{(k)}(r(f_1^{(k)}))>0$.
Then we see that
$f_3^{(k)}(r(f_1^{(k)}))=-\epsilon_3 f_2^{(k)}(r(f_1^{(k)}))<0$, which implies that $r(f_3^{(k)})>r(f_1^{(k)})$.
The assumption that $c_{n-1}^{(2)}=-1$ also gives $r(f_2^{(k)})\ge r(f_2^{(n-1)})=1$,
so $r(f_3^{(k)})>1$. Choose
\[
\delta_{\epsilon_3}:=\frac{1}{2}\min\{1,
\min_{1\le k<n}\frac{r(f_3^{(k)})-r(f_1^{(k)})}{r(f_3^{(k)})+2}\}.
\]
Then $0<\delta_{\epsilon_3}<1$ and
\begin{equation}\label{eq:f3-dilated-root-gap}
 (1-\delta_{\epsilon_3})r(f_3^{(k)}) \ge r(f_1^{(k)})+2\delta_{\epsilon_3}.
\end{equation}

Define the right-Noetherian polynomial
$f_5(x)=f_3(\frac{x+2\delta_{\epsilon_3}}{1-\delta_{\epsilon_3}})$, and $g_5=\Pol_n f_5$.
Then \[
r(f_5^{(k)})=(1-\delta_{\epsilon_3})r(f_3^{(k)})-2\delta_{\epsilon_3} \ge r(f_1^{(k)})
\]
for all $1\le k<n$.
By Definition \ref{Dominance}, $f_1'\triangleleft f_5'$. By Theorem \ref{DominancePreserve}, $\Pol_{n-1}f_1'\triangleleft \Pol_{n-1}f_5'$.
So for all $i=1,\ldots,n$,
\[
\Upsilon_{\partial_i g_5}=\pi_i^{-1}\Upsilon_{\Pol_{n-1}f_5'}\subset \pi_i^{-1}\Upsilon_{\Pol_{n-1}f_1'}=\Upsilon_{\partial_i g_1},
\]
where \[
\pi_i(\lambda_1,\ldots,\lambda_n)=(\lambda_1,\ldots,\lambda_{i-1},\lambda_{i+1},\ldots,\lambda_n).
\]
Taking the intersection provides
\[
(1-\delta_{\epsilon_3})\Upsilon_{g_3}^{\,1}-2\delta_{\epsilon_3}\mathbf1 =\Upsilon_{g_5}^{\,1} \subset \Upsilon_{g_1}^{\,1}.\]
Since $\Upsilon_{g_1}^{\,1}$ passes PRT, we have
\[(1-\delta_{\epsilon_3})\overline{\Upsilon_{g_3}^{\,1}}\subset
2\delta_{\epsilon_3}\mathbf1+\overline{\Upsilon_{g_1}^{\,1}}
\subset\Upsilon_{g_1}^{\,1}+\delta_{\epsilon_3}\mathbf1,\] which finishes the proof.
\end{proof}

In order to study the coerciveness problem, let us first prove a lemma.
\begin{lem}
\label{d1-sup-comparison}
\begin{enumerate}
    \item $d_1(\sup_X\varphi,\varphi)\le V_2\sup_X\varphi$ and $d_1(0,\varphi)\le 2V_2\sup_X\varphi$ for all $\varphi\in\HH$ such that $I_2(\varphi)=0$.
    \item For any $\kappa>0$, there exist a constant $\delta>0$ and another constant $C$ independent of $\varphi$ but possibly dependent on $\kappa$ such that $d_1(0,\varphi)\ge \delta \sup_X\varphi-C$ for all $\varphi\in\HH$ such that $I_2(\varphi)=0$ and $g_2(\chi_\varphi)\ge \kappa\omega^n$.
    \item If $I_2(\varphi)=0$, then $I_2(\varphi-\sup_X\varphi)=-V_2 \sup_X\varphi$.
\end{enumerate}
\end{lem}
\begin{proof}
Assume that $I_2(\varphi)=0$. Let $s=\sup_X\varphi$ and $u=\varphi-s\le0$. First,
\begin{equation}
\label{eq:I2-positive-volume}
I_{2}(\varphi)=\int_{X}\varphi\left(\int_{0}^{1}g_{2}(\chi_{t\varphi})dt\right)
\end{equation}
by \eqref{functional properties}. Since $\mathcal{H}$ is convex and both $0,\varphi\in \mathcal{H}$, we have $\int_{0}^{1}g_{2}(\chi_{t\varphi})dt>0$ pointwise. Hence $s\ge 0$. Then using the path $\varphi_{s,t}=(1-t)s+t\varphi$ such that $\frac{d\varphi_{s,t}}{dt}=\varphi-s\le 0$, we have
\begin{equation}
\label{d1svarphi}
\begin{split}
d_1(s,\varphi)\le\int_0^1 \int_X|\frac{d\varphi_{s,t}}{dt}|g_2(\chi_{\varphi_{s,t}})\,dt=-\int_0^1 \int_X \frac{d\varphi_{s,t}}{dt}g_2(\chi_{\varphi_{s,t}})\,dt\\
=-\int_0^1\frac{dI_2(\varphi_{s,t})}{dt}\,dt=I_2(s)-I_2(\varphi)=I_2(s)=V_2 s.
\end{split}
\end{equation}
Using the path $t\in[0,1]\to ts$, we see that
$d_1(s,0)\le V_2s$. So $d_1(0,\varphi)\le 2V_2 s$ by the triangle inequality. This finishes the first part.

Next, using Theorem $1.3(2)$ of \cite{ChenGhoshNieGeodesics}, we have $$\kappa\int_{X}\varphi_{+} \omega^n\le\int_{X}\varphi_{+}g_{2}(\chi_\varphi)\le d_{1}(0,\varphi).$$
We know that $u$ is $\chi$-psh  and $\sup_{X}u=0$. This gives 
\[
    \int_{X}(-u)\chi^{n}\le C\implies s\int_{X}\chi^{n}\le C+\int_{X}\varphi\chi^{n}\le C+\int_{X}\varphi_{+}\chi^{n}\le C+C\int_{X}\varphi_{+}\omega^{n}.
\] 
This finishes the second part. The third part follows from \eqref{functional properties} along $\varphi-t\sup_X\varphi$.
\end{proof}

Define $I_3, I_4$ by the same first-variation formula as $I_1,I_2$, with $I_3(0)=I_4(0)=0$ and put $J_3=-I_3$.
Linearity in the polynomial coefficients gives
\begin{equation}\label{eq:f3-f4-energy-identities}
 \begin{split}
 J_3(\varphi)
   &=J_1(\varphi)+\epsilon_3 I_2(\varphi)
       -\epsilon_3 v_2\int_X\varphi\,\omega^n,\\
 I_4(\varphi)&=I_2(\varphi)-\epsilon_4 J_1(\varphi).
 \end{split}
\end{equation}
In particular, $V_3:=\int_Xg_3(\chi_\varphi)=V_1=0$ and
$V_4:=\int_Xg_4(\chi_\varphi)=V_2>0$.
Thus the perturbed equation has exactly the required
integral compatibility in the original class, and
$J_3$ is invariant under adding constants.

Now we are ready to prove the perturbation lemma.
\begin{lem}\label{lem:f3-f4-coercivity-transfer}
Assume the coercivity inequality for $J_1$ in
Theorem~\ref{thm:main}. Fix $\epsilon_4$ as above. After
decreasing $\epsilon_{3,0}(\epsilon_4)$ if necessary, $J_3$
is coercive on $\HH_4$ with respect to its $g_4$ path
distance $d_{1,4}$, on the slice $I_4=0$.
\end{lem}
\begin{proof}
Let $u\in\HH_4$ with $\sup_X u=0$. By \eqref{eq:H4-subset-H2}, $F_\omega(\chi_u)>-1/\epsilon_4$. Proposition~\ref{prop:comparison} then gives
$g_2(\chi_u)\ge\kappa_{\epsilon_4}\omega^n$. By Lemma \ref{d1-sup-comparison}, the coercivity implies
\begin{equation}\label{eq:f4-restricted-energy-control}
 J_1(u) \ge -bI_2(u)-C
 \quad\text{for some }b>0\text{ depending on the fixed }\epsilon_4
\end{equation}
since $J_1$ is invariant under adding constants.
Choose also $\epsilon_3<b$ and put $\delta=\frac{b-\epsilon_{3}}{1+b\epsilon_{4}}>0$.
Since $u\le0$, \eqref{eq:f3-f4-energy-identities} and
\eqref{eq:f4-restricted-energy-control} give
\begin{equation}\label{eq:f4-coercive-energy-bound}
 \begin{split}
 &J_3(u)\ge J_{1}(u)+\epsilon_{3}I_{2}(u)= (1-\delta \epsilon_4) J_{1}(u) + \delta \epsilon_4 J_{1}(u)+\epsilon_{3}I_{2}(u)\\
 &\ge b(1-\delta \epsilon_4)(-I_{2}(u))+ \delta \epsilon_4 J_{1}(u)+\epsilon_{3}I_{2}(u)-C= \delta(-I_{4}(u))-C.
 \end{split}
\end{equation}
This implies the required coerciveness by Lemma \ref{d1-sup-comparison} applied to $f_4$ since $J_3$ is also invariant under adding constants.
\end{proof}

\section{Short time existence of the flow}

The following lemma is the analogue of
\cite[Lemma 2.4]{ChuLeeHypercriticalDHYM} for the operator in \eqref{eq:flow}.

\begin{lem} \label{Ellipticity and Concavity}
Define
\[
 F(\lambda)=\frac{g_1(\lambda)}{g_2(\lambda)},\qquad
 F_i=\frac{\partial F}{\partial\lambda_i},\quad
 F_{ij}=\frac{\partial^2F}{\partial\lambda_i\partial\lambda_j}.
\]
Then $\Upsilon_{g_2}$ is an open convex symmetric set, and the following
properties hold:
\begin{enumerate}
\item $F_i>0$ on $\Upsilon_{g_2}$ for every $i=1,\ldots,n$.
\item $F$ is concave on $\Upsilon_{g_2}$.
\item If $\lambda_i\ge\lambda_j$, then $F_i\le F_j$ on $\Upsilon_{g_2}$.
\item For every positive definite
Hermitian matrix $A$, the operator
\[
F_A(B)=F(\lambda(A^{-1}B))=F(\lambda(A^{-1/2}BA^{-1/2}))
\]
is elliptic on the convex set
\[\Upsilon_{g_2}(A)=\{B\in\operatorname{Herm}(n):
              \lambda(A^{-1}B)\in\Upsilon_{g_2}\},\]
i.e. $\sum_{i, j=1}^{n}F^{i\bar j}\xi_i\bar\xi_j>0$ for all $B\in\Upsilon_{g_2}(A)$ and nonzero $\xi\in\mathbb{C}^n$, where $F^{i\bar j}=\frac{\partial}{\partial {B_{i\bar j}}}F_A(B)$.
\item $F_A(B)$ is concave on $\Upsilon_{g_2}(A)$.
\end{enumerate}
\end{lem}
\begin{proof}
\begin{enumerate}
    \item $\partial_{\lambda_i}F(\lambda)>0$ by Proposition \ref{Strict proper position}.
    \item By Corollary \ref{polarizationproperposition}, $g_2\idealprec g_1$. Then $F=g_1/g_2$ is concave on $\Upsilon_{g_2}$ by Theorem \ref{ConcaveQuotient}.
    \item By Lemma \ref{FangMaLinCorrespondence} and Theorem \ref{kappa-polarization}, $g_2$ is an ideal Fang-Ma-G\aa rding polynomial. So $\Upsilon_{g_2}$ is convex by definition. Now let $\mu$ be obtained from $\lambda$ by interchanging $\lambda_i$ and $\lambda_j$; then $t\mu+(1-t)\lambda\in\Upsilon_{g_2}$ for all $t\in[0,1]$. By the concavity of the function $t\to F(t\mu+(1-t)\lambda)$, 
\[
 0=F(\mu)-F(\lambda)
 \le \frac{d}{dt}\vert_{t=0} F(t\mu+(1-t)\lambda)
 =-(F_i-F_j)(\lambda_i-\lambda_j).
\]
Thus $F_i\le F_j$ when $\lambda_i>\lambda_j$.
When $\lambda_i=\lambda_j$, the symmetry gives $F_i=F_j$.
\item Without loss of generality, we assume that $A=I$. Then $\Upsilon_{g_2}(A)$ is convex by Proposition \ref{convexmatrix}. At a diagonal matrix $B=\operatorname{diag}(\lambda_1,\ldots,
\lambda_n)$, by Theorem 1.4 of \cite{SpruckGeometricAspects}, $F^{i\bar j}=\delta_{ij}F_i$. The ellipticity now follows from (1).
\item When $\lambda_i$ are distinct, by Theorem 1.4 of \cite{SpruckGeometricAspects}, \begin{equation}\label{eq:flow-matrix-concavity}
 F^{i \bar j, r\bar s} H_{i \bar j}H_{r\bar s}=\frac{\partial^2 F_A(B)}{\partial {B_{i\bar j}}\partial {B_{r\bar s}}}H_{i \bar j}H_{r\bar s}
 =\sum_{i,j}F_{ij}H_{i\bar i}H_{j\bar j}
   +2\sum_{i<j}\frac{F_i-F_j}{\lambda_i-\lambda_j}
                    |H_{i\bar j}|^2\le 0
\end{equation}
by (2) and (3). This also holds at repeated eigenvalues by continuity.
\end{enumerate}
\end{proof}

Now we are ready to prove the short-time existence of the flow.

\begin{prop}\label{prop:short-time}
For every $\varphi_0\in\HH$, there is $T>0$ such
that \eqref{eq:flow} has a unique smooth solution on $X\times[0,T]$
with $\varphi(t)\in\HH$. On every interval of smooth admissible
existence, it satisfies
\begin{equation}\label{eq:short-time-speed}
 \min_X F_\omega(\chi_{\varphi_0})
 \le\frac{\partial\varphi}{\partial t}
 \le\max_X F_\omega(\chi_{\varphi_0}).
\end{equation}
\end{prop}
\begin{proof}
The short-time existence of a solution $\varphi(t)\in\HH$ is standard by the ellipticity of $F_\omega$ and the openness of $\HH$. The uniqueness is also standard by the convexity of $\HH=\Upsilon_{g_2}(\omega)$. 

Finally, differentiating \eqref{eq:flow} in time gives
\[
 (\frac{\partial}{\partial t}-L_\varphi)\frac{\partial\varphi}{\partial t}=(\frac{\partial}{\partial t}-F_\omega^{i\bar j}\partial_i\bar\partial_j)\frac{\partial\varphi}{\partial t}=0.
\]
The maximum principle proves \eqref{eq:short-time-speed}.
\end{proof}

\section{Long-time existence of the flow}
\label{sec:long-time}

In this section, we prove the long-time existence of the flow using an argument similar to that in Lemma~4.3 and Theorem~4.4 of Chu-Lee's paper \cite{ChuLeeHypercriticalDHYM}.

Let $[0,T_{\max})$ be the maximal interval of smooth admissible
existence given by Proposition~\ref{prop:short-time}. We start with the $C^0$ bound.
\begin{lem}
\label{lemma:C^{0}}
Set $C_*:=\|F_\omega(\chi_{\varphi_0})\|_{L^\infty(X)}$. Then $\|\varphi(t)\|_{L^\infty}\le\|\varphi_0\|_{L^\infty}+C_*t$.
\end{lem}
\begin{proof}
 Proposition~\ref{prop:short-time} gives
$|\dot\varphi|\le C_*$, and hence
$\|\varphi(t)\|_{L^\infty}\le\|\varphi_0\|_{L^\infty}+C_*t$.
In particular, the potential is bounded on every finite time interval.
\end{proof}

Now we prove a lemma needed later.

\begin{lem}\label{lem:flow-linearization-upper}
There is a constant $A>0$, depending only on $C_*$, $\kappa:=\kappa_{C_*}>0$,
$n$, and the coefficients of $g_1,g_2$, such that along the flow
\[
 0<F_i(\lambda)\le\frac{A}{\lambda_i}\le A,
 \qquad
 \lambda=\lambda(\omega^{-1}\chi_\varphi).
\]
\end{lem}
\begin{proof}
By \eqref{eq:denominator-cone-normalization}, each $\lambda_i>1$.
By Proposition \ref{prop:comparison}, $g_2(\lambda)\ge\kappa \prod_{j=1}^n\lambda_j\ge\kappa$ along the flow.
 For $a=1,2$, \[
 |\partial_i g_a(\lambda)|=|\sum_{k=1}^n c_k^{(a)}\sigma_{k-1}(\lambda_j, j\not=i)|\le \sum_{k=1}^n |c_k^{(a)}|\binom{n-1}{k-1}\prod_{j\ne i}\lambda_j
\]
using $\lambda_j>1$ for all $j$. So using $|F|\le C_*$,
\[
F_i=\frac{\partial_i g_1-F\partial_i g_2}{g_2}\le\frac{A' \prod_{j\ne i}\lambda_j}{\kappa \prod_{j=1}^n\lambda_j}=\frac{A}{\lambda_i}
\] for a constant $A'$, and $A=\frac{A'}{\kappa}$.
\end{proof}

Now we are ready to prove the complex Hessian estimate using a similar argument to that in Theorem 4.4 of \cite{ChuLeeHypercriticalDHYM}.

\begin{prop}\label{prop:flow-finite-time-trace}
There is a constant $C>0$, depending only on the fixed data and
$\varphi_0$, such that
\begin{equation}\label{eq:flow-finite-time-trace}
 \sup_X\operatorname{tr}_\omega\chi_{\varphi(t)}
 \le e^{Ct}\sup_X\operatorname{tr}_\omega\chi_{\varphi_0}
 \qquad (0\le t<T_{\max}).
\end{equation}
Consequently, for every finite $T\le T_{\max}$, there is $C_T>0$ such that $0<\lambda_i(x,t)\le C_T$ for
$0\le t< T$.
\end{prop}
\begin{proof}
Write $L_\varphi=F_{\omega}^{i\bar j}\nabla_i\nabla_{\bar j}$ for the connection $\nabla_i$ with respect to $\omega$. By the same calculation as in Lemma 4.3 of \cite{ChuLeeHypercriticalDHYM}, we get
$$(\partial_{t}-L_\varphi)(\chi_{\varphi})_{i\bar{j}}=F_{\omega}^{p\bar{q},k\bar{l}}\nabla_{i}\chi_{\varphi,p\bar{q}}\nabla_{\bar{j}}\chi_{\varphi,k\bar{l}}+F_{\omega}^{p\bar{q}}\big(R_{p\bar{j}i}^{k}\chi_{\varphi,k\bar{q}}-R_{p\bar{j}\ \bar{q}}^{\ \ \ \bar{l}}\chi_{\varphi,i\bar{l}}\big).$$
Taking its $\omega$-trace gives
\[
(\partial_t-L_\varphi)\operatorname{tr}_\omega\chi_\varphi=g^{i\bar{j}}F_{\omega}^{p\bar{q},k\bar{l}}\nabla_{i}\chi_{\varphi,p\bar{q}}\nabla_{\bar{j}}\chi_{\varphi,k\bar{l}}+g^{i\bar{j}}F_{\omega}^{p\bar{q}}\big(R_{p\bar{j}i}^{k}\chi_{\varphi,k\bar{q}}-R_{p\bar{j}\ \bar{q}}^{\ \ \ \bar{l}}\chi_{\varphi,i\bar{l}}\big).
\]
By Lemma \ref{Ellipticity and Concavity}(5),
$g^{i\bar{j}}F_{\omega}^{p\bar{q},k\bar{l}}\nabla_{i}\chi_{\varphi,p\bar{q}}\nabla_{\bar{j}}\chi_{\varphi,k\bar{l}}\le 0$. By Lemma \ref{lem:flow-linearization-upper}, \[g^{i\bar{j}}F_{\omega}^{p\bar{q}}\big(R_{p\bar{j}i}^{k}\chi_{\varphi,k\bar{q}}-R_{p\bar{j}\ \bar{q}}^{\ \ \ \bar{l}}\chi_{\varphi,i\bar{l}}\big)\le C_\omega|\chi_\varphi|_\omega\le C \operatorname{tr}_\omega\chi_\varphi.\]

The maximum principle applied to $e^{-Ct}\operatorname{tr}_\omega\chi_\varphi$ on
$X\times[0,T]$, for finite $T<T_{\max}$, proves
\eqref{eq:flow-finite-time-trace}. Since
$\lambda_i\le\operatorname{tr}_\omega\chi_\varphi$, it also gives the required eigenvalue bound.
\end{proof}

As a corollary, \[\lambda(\omega^{-1}\chi_\varphi)\in\{\lambda\in\overline{\Upsilon_{g_2}}:
1\le\lambda_i\le C_T\text{ for all }i,\ g_2(\lambda)\ge\kappa\}\subset\Upsilon_{g_2}.\] By compactness, we get the uniform ellipticity and concavity of $F$ on $X\times[0,T)$ by Lemma \ref{Ellipticity and Concavity} for any finite $T\le T_{\max}$. Then the standard parabolic Evans-Krylov estimate implies the bound on the $C^{\alpha}$ space-time semi-norm of $\ddc\varphi$. The standard parabolic Schauder estimate implies the space-time $C^{k,\alpha}$ estimates. Then the Arzel\`a–Ascoli theorem, together with the short-time existence, implies the long-time existence.

\section{Coerciveness implies existence}

In this section, we prove that the coerciveness of the $J_1$ functional on the space $\mathcal{H}$ with $I_2=0$ implies the existence. 

We first generalise Definition 3.3 of \cite{ChenJEquationDHYM} to the ideal Fang-Ma-G{\aa}rding setting:

\begin{defn}
\label{DefnDegenerateConeCondition}
Suppose $\omega$ is a K\"{a}hler form, and $\Theta$ is a closed, positive $(1,1)$ current. Suppose that $\Upsilon\subset\Gamma_n^+$ is convex, invariant under
coordinate permutations, and passes PRT. Let $\bar\Upsilon$ be its closure. We say that $\Theta\in\bar\Upsilon(\omega)$ if for any coordinate chart $U$ with local potential $\Theta|_U=\sqrt{-1}\partial\bar\partial\varphi_U$, we have $\sqrt{-1}\partial\bar\partial\varphi_{U,r}\in \bar\Upsilon(\omega_U)$ on $U_r:=\{x\in X:B(x,r)\subset U\}$ for any K\"{a}hler metric $\omega_U$ on $U$ with constant coefficients satisfying $\omega_U\leq\omega$, where $\varphi_{U,r}$ is the convolution of $\varphi_U$ with the mollifier $\rho_r$.
\end{defn}

Now we prove a useful lemma.
\begin{lem}
\label{lem:useful lemma}
    For any $\varphi_{0},\varphi_{1},\varphi_{2}\in \mathcal{H}$ satisfying $\varphi_{0}\le\varphi_{1}\le\varphi_{2}$ we have $$d_{1}(\varphi_{0},\varphi_{1})\le d_{1}(\varphi_{0},\varphi_{2}).$$
    \begin{proof}
        For $i=0,1,2$, let $\varphi_{\epsilon,i}$ be the $\epsilon$-geodesic between $\varphi_{0}$ and $\varphi_{i}$ as in \cite{ChenGhoshNieGeodesics}. We also denote the corresponding $L^{1}$-energy by \[E_{1}^{\epsilon,i}(t)=\int_X |\frac{d}{dt}\varphi_{\epsilon,i}(t)|g_2(\chi_{\varphi_{\epsilon,i}(t)}).\] Then by the comparison principle,
        $$\varphi_{\epsilon,0}\le \varphi_{\epsilon,1}\le\varphi_{\epsilon,2}.$$
        This implies 
        $$\frac{d}{dt}\varphi_{\epsilon,0}(0)\le\frac{d}{dt} \varphi_{\epsilon,1}(0)\le\frac{d}{dt}\varphi_{\epsilon,2}(0).$$
        Now using Proposition 4.3 of \cite{ChenGhoshNieGeodesics}, we have 
        \[
        |\frac{d}{dt}\varphi_{\epsilon,0}(0)|\le C\epsilon^{2} \implies |\frac{d}{dt}\varphi_{\epsilon,1}(0)|\le |\frac{d}{dt}\varphi_{\epsilon,2}(0)| + C\epsilon^{2} \implies
        E_{1}^{\epsilon,1}(0)\le E_{1}^{\epsilon,2}(0)+C\epsilon^{2}.\]
        Now letting $\epsilon\rightarrow 0$ and using Theorem $8.3(2)$ of \cite{ChenGhoshNieGeodesics}, we obtain
        $$d_{1}(\varphi_{0},\varphi_{1})\le d_{1}(\varphi_{0},\varphi_{2}).$$
        This completes the proof.
    \end{proof}
\end{lem}
Now we are ready to prove the following proposition using the flow:
\begin{prop}
Suppose that $f_1$ is strictly right-Noetherian and  $f_2$ is right-Noetherian. $f_2\prec f_1$ strongly strictly, and $f_2\idealprec f_1$. Then if $\int_X g_1(\chi_\varphi)=0$ and $J_1(\varphi)\ge \delta d_1(0,\varphi)-C$ for all $\varphi\in\mathcal{H}=\mathcal{H}_2$ such that $I_2(\varphi)=0$, there exists a current $\chi_\varphi\in\bar\Upsilon_{g_1}(\omega)$ with zero Lelong number.
\label{coercive-Lelong-0}
\end{prop}

\begin{proof}
Normalize the initial potential
by $I_2(\varphi_0)=0$. Along \eqref{eq:flow},
\[
 \frac{d}{dt}I_2(\varphi_t)=\int_Xg_1(\chi_{\varphi_t})=0,
 \qquad
 \frac{d}{dt}J_1(\varphi_t)
       =-\int_X F_\omega(\chi_{\varphi_t})^2g_2(\chi_{\varphi_t})\le0.
\]
Coercivity therefore yields $d_1(0,\varphi_t)\le C$. We have used \eqref{functional properties} in the above calculation.

Let $s_t=\sup_X\varphi_t$ and $u_t=\varphi_t-s_t\le0$. By Lemma \ref{d1-sup-comparison}, $s_t\le C$ and \[d_1(0,u_t)=d_1(s_t,\varphi_t)\le C.\] By Lemma \ref{lem:useful lemma} in the second inequality, Theorem $1.3(2)$ of \cite{ChenGhoshNieGeodesics} in the third inequality, the convexity of $\mathcal{H}$, the concavity of $F$, and Proposition \ref{prop:comparison} in the fourth inequality, and the inequality $(\chi+\chi_{u_t})^n\ge \chi_{u_t}^n$ in the fifth inequality, we see that
\[
 C\ge d_1(u_t,0)\ge d_1(u_t,u_t/2)\ge\frac12\int_X|u_t|g_2(\chi_{u_t/2})\ge\frac{\kappa_C}{2}\int_X|u_t|\chi_{u_t/2}^n
      \ge\kappa_C2^{-n-1}\int_X|u_t|\chi_{u_t}^n.
\]
 We know that $\{u_{t}\in \Psh(X,\chi):\sup_{X}u_{t}=0\}$ is compact in the weak topology of $L^{1}(X,\chi)$.  Consequently  every subsequential weak $L^1$ limit $u_\infty$ belongs to the \emph{classical}
space $\mathcal E^1(X,\chi)$ and has zero Lelong numbers by Corollaries 2.7 and 1.8 of \cite{GuedjZeriahiFiniteEnergy}.

On a long-time solution, the lower bound for $J_1$ also gives
\[
 \int_T^\infty\int_X
       F_\omega(\chi_{\varphi_t})^2g_2(\chi_{\varphi_t})\,dt
       \longrightarrow0\qquad(T\to\infty).
\]
Hence there exist $t_j\to\infty$ such that, writing $u_j=u_{t_j}$,
\[
 \int_XF_\omega(\chi_{u_j})^2g_2(\chi_{u_j})\longrightarrow0
\implies \int_XF_\omega(\chi_{u_j})^2\omega^n\longrightarrow0.
\]
The implication uses $g_2(\chi_{u_j})\ge\kappa_C\omega^n$.

Now for any coordinate chart $U$ with local potential $\chi|_U=\sqrt{-1}\partial\bar\partial\varphi_{\chi,U}$, and any K\"{a}hler metric $\omega_U$ on $U$ with constant coefficients satisfying $\omega_U\leq\omega$, we have \[F_{\omega_U}(\sqrt{-1}\partial\bar\partial(\varphi_{\chi,U}+u_j))=F_{\omega_U}(\chi_{\varphi_j})\ge F_{\omega}(\chi_{\varphi_j})\] by Proposition \ref{Strict proper position} and the Courant-Fischer-Weyl min-max principle. Let $\rho_r$ be the mollifier; by the concavity of $F_{\omega_U}$ and the Cauchy-Schwarz inequality, we get
\[
 F_{\omega_U}(\ddc((\varphi_{\chi,U}+u_j)*\rho_r))
 \ge F_{\omega_U}(\chi_{\varphi_j})*\rho_r \ge F_{\omega}(\chi_{\varphi_j})*\rho_r
 \ge-C_r\|F_{\omega}(\chi_{\varphi_j})\|_{L^2}
\]
on
$U_r:=\{x\in X:B(x,r)\subset U\}$. Letting $j\to\infty$ and taking the limit, we see that \[
 F_{\omega_U}(\ddc((\varphi_{\chi,U}+u_\infty)*\rho_r))\ge 0.
\] So $\chi_{u_\infty}\in\bar\Upsilon_{g_1}(\omega)$ in the sense of Definition \ref{DefnDegenerateConeCondition}.
\end{proof}

Now we start to use the perturbation technique. Choose $\epsilon_4$ such that for $f_4=f_2(x)+\epsilon_4 f_1(x)$, the space 
\[
 \HH_4=\{\varphi\in\HH_2:
               1+\epsilon_4 F_\omega(\chi_\varphi)>0\}
 \subseteq\HH_2
\] is non-empty. Without loss of generality, assume that $0\in\HH_4$. Then we choose $\epsilon_3$ sufficiently small that \[f_3=f_1(x)-\epsilon_3f_2(x)+\epsilon_3 v_2\] and $f_4$ satisfy all the conditions in Prop \ref{coercive-Lelong-0}. Then we get a current $\chi_{u_\infty}\in\bar\Upsilon_{g_3}(\omega)$ with zero Lelong number. Now we take sufficiently small charts $U$ such that there exists a K\"ahler form $\omega_U$ with constant coefficients such that $(1-\delta_{\epsilon_3})\omega\le\omega_U\le\omega$ on $U$, where $\delta_{\epsilon_3}$ is the constant in Lemma \ref{lem:f3-strict-cone-inclusion}. Then the local regularization method in \cite{BlockiKolodziejRegularization} provides a smooth potential $\underline{\varphi}$ in the class $[\chi+\delta_{\epsilon_3}\omega]$ such that
\[\chi+\delta_{\epsilon_3}\omega+\ddc\underline{\varphi}\in \bar\Upsilon_{g_3}(\omega_U)\subset \bar\Upsilon_{g_3}((1-\delta_{\epsilon_3})\omega)= (1-\delta_{\epsilon_3})\bar\Upsilon_{g_3}(\omega).\]
So \[\chi+\ddc\underline{\varphi}\in (1-\delta_{\epsilon_3})\bar\Upsilon_{g_3}(\omega)-\delta_{\epsilon_3}\omega\subset(1-\delta_{\epsilon_3})\bar\Upsilon^1_{g_3}(\omega)-\delta_{\epsilon_3}\omega\subset\Upsilon^1_{g_1}(\omega)\] by Lemma \ref{lem:f3-strict-cone-inclusion}. Finally, we solve the equation $g_1(\chi_\varphi)=0$ using Lin's existence result \cite{LinInverseSigmaSolvability}.

\section{Existence implies coercivity}

In this section, we assume that $\varphi_*\in\mathcal{H}$ solves the equation $g_1(\chi_{\varphi_*})=0$ such that $\varphi_*$ is normalized by $I_2(\varphi_*)=0$. By the triangle inequality of the $d_1$ distance, we can, without loss of generality, assume that $\varphi_*=0$.

We first prove another useful perturbation result:

\begin{lem}\label{lem:derivative-perturbation}
There exists $\eta>0$ such that
$f_2\prec f_1+\eta f_2'$.
\end{lem}
\begin{proof}
For $0\le k<n$, put $r_k=r(f_2^{(k)})$; then by Proposition~\ref{Strict proper position}, there exists $q_k<m_k$ such that
\[f_2^{(k)}(x)\sim \frac{f_2^{(k+m_k)}(r_k)}{m_k!}(x-r_k)^{m_k}\] and \[f_1^{(k)}(x)\sim \frac{f_1^{(k+q_k)}(r_k)}{q_k!}(x-r_k)^{q_k}\] for $x\downarrow r_k$, where the coefficients $f_2^{(k+m_k)}(r_k)>0$ and $f_1^{(k+q_k)}(r_k)<0$. Consequently, for a sufficiently small common $\eta>0$, each derivative of $f_1+\eta f_2'$ satisfies
\begin{equation}\label{eq:derivative-perturbation-sign}
 f_1^{(k)}(x)+\eta f_2^{(k+1)}(x)<0\quad\text{for }r_k<x<r_k+\delta_k,
 \qquad 0\le k<n,
\end{equation}
for some $\delta_k>0$. Then we can apply the same induction argument as in the proof of \eqref{eq:strict-proper-perturbation} in Proposition \ref{Strict proper position} to prove that $f_1+\eta f_2'$ is right-Noetherian, and $f_2\prec f_1+\eta f_2'$.
\end{proof}

Let $g_{2}'(\chi_\varphi)$ be the formal derivative of $g_{2}(\chi_\varphi)$ as in \cite{ChenGhoshNieGeodesics}. It is positive on $\mathcal{H}$, by the derivative positivity in \cite[Lemma 5.1]{FangMaGarding}, or
\cite[Lemma 2.19]{ChenGhoshNieGeodesics}.
Define the analogy of Aubin's $I-J$ functional by
\begin{equation}\label{eq:based-aubin-energy}
 \mathcal A(\varphi)=I_2(\varphi)-\int_X \varphi\,g_2(\chi_\varphi).
\end{equation}
This functional vanishes at $\varphi=0$ and is invariant under adding constants. 
\begin{lem}\label{lem:simple}
    If $q(t)$ is a non-negative polynomial on $[0,1]$ of degree at most $n-1$, then there is a number $a_{n}>0$ independent of $q$ such that 
    \[\int_{0}^{1}tq(t)dt\ge a_{n}\int_{0}^{1}q(t)dt.\]
    \end{lem}
    \begin{proof}
       After normalization, assume that $\int_{0}^{1}q(t)dt= 1$. Note that the space of polynomials of degree at most $n-1$ is finite-dimensional, and any two norms on a finite-dimensional space are equivalent. We therefore have a constant $C_{n}>0$ such that 
        \[\|q\|_{L^{\infty}([0,1])}\le C_{n}\|q\|_{L^{1}([0,1])} \le C_{n}.\]
        Fix any $\delta>0$; then 
        \begin{align*}
            \int_{0}^{1}tq(t)dt\ge \int_{\delta}^{1}tq(t)dt\ge \delta\int_{\delta}^{1}q(t)dt= \delta(1-\int_{0}^{\delta}q(t)dt)\ge \delta(1-\delta C_{n}).
        \end{align*}
        Now choose $\delta=\min\{\frac{1}{2},\frac{1}{2C_{n}}\}$. Then $1-\delta C_{n}\ge \frac{1}{2}$. This implies the required estimate 
        \[\int_{0}^{1}tq(t)dt\ge \frac{\delta}{2}>0.\]
    \end{proof}

We first prove the coerciveness of $\mathcal{A}(\varphi)$.

\begin{lem}\label{lem:based-energy-coercivity}
There exist $a>0$ and $b\ge0$, depending only on the fixed data
and $\varphi_*$, such that
\begin{equation}\label{eq:based-energy-coercivity}
 \mathcal A(\varphi)\ge a\,d_1(0,\varphi)-b
 \quad\text{whenever }I_2(\varphi)=0.
\end{equation}
Moreover, $\mathcal A(\varphi)\ge0$ for every $\varphi\in\HH$.
\end{lem}
\begin{proof}
By convexity of $\HH$, the affine segment $t\varphi\in \HH$, $0\le t\le1$. By \eqref{functional properties} and its derivative, we have 
\[\frac{d}{dt}(I_2(t\varphi))=\int_{X}\varphi g_{2}(\chi_{t \varphi}), \quad q_{\varphi}(t):=-\frac{d^{2}}{dt^{2}}(I_2(t\varphi))=\int_{X}\sqrt{-1}\partial \varphi \wedge \bar{\partial}\varphi\wedge g'_{2}(\chi_{t \varphi})\ge 0.\]
The function $q_\varphi$ is a polynomial of degree at most $n-1$.
Integration by parts in $t$ gives
\begin{equation}\label{eq:based-energy-moments}
 I_{2}(\varphi)
 =\int_{0}^{1}\frac{dI_2(t\varphi)}{dt}\,dt=(t\frac{dI_2(t\varphi)}{dt})|_{t=0}^{1}-\int_{0}^{1}t\frac{d^2I_2(t\varphi)}{dt^2}dt=\int_{X}\varphi g_{2}(\chi_{\varphi})+\int_{0}^{1}tq_{\varphi}(t)\,dt.
 \end{equation}
 So by Lemma \ref{lem:simple}, there is $a_n>0$ such that \[
 \mathcal A(\varphi)=I_{2}(\varphi)-\int_{X}\varphi g_{2}(\chi_{\varphi})=\int_0^1 t q_\varphi(t)\,dt \ge a_n\mathcal I(\varphi),\]
 where
 \[\mathcal I(\varphi):=\int_0^1 q_\varphi(t)\,dt=\int_X \varphi \bigl(g_2(\chi)-g_2(\chi_\varphi)\bigr)\]
 is the analogy of Aubin's I-functional. This should not be confused with the Aubin-Mabuchi mixed energy function $I_2$.

Now assume $I_2(\varphi)=0$ and put $s=\sup_X \varphi$. By \eqref{eq:I2-positive-volume}, $s\ge0$.
Since $\varphi-s\le 0$ is a $\chi$-plurisubharmonic
function, and the fixed volume form $g_2(\chi)$ is comparable to $\chi^n$, we see that
\[\int_X \varphi\,g_2(\chi)-V_2s=\int_X (\varphi-s)g_2(\chi)\ge C\int_X (\varphi-s)\chi^n \ge -C.\]
So
\[\mathcal I(\varphi)=\int_X \varphi\,g_2(\chi)+\mathcal A(\varphi)\ge \int_X \varphi\,g_2(\chi) \ge V_2 s-C.\]
This finishes the proof by Lemma \ref{d1-sup-comparison}.
\end{proof}

For $\varepsilon>0$, we perturb the based functional by
\begin{equation}\label{eq:solution-based-perturbation}
 \mathcal J_\varepsilon(\varphi)=J_1(\varphi)-\varepsilon\mathcal A(\varphi),
 \qquad
 g_{1,\varepsilon}(\chi_\varphi)
 =g_1(\chi_\varphi)-\varepsilon \ddc\varphi\wedge g_{2}'(\chi_\varphi).
\end{equation}
Along any smooth path $\varphi_t\in\HH$,
\begin{equation}
    \begin{split}
        \frac{d}{dt}\mathcal{A}(\varphi_t)&=\frac{d}{dt}I_2(\varphi_t)-\frac{d}{dt}\int_X \varphi_t g_2(\chi_{\varphi_t})\\
        &=\int_X \frac{d\varphi_t}{dt}g_2(\chi_{\varphi_t})-\int_X \frac{d\varphi_t}{dt}g_2(\chi_{\varphi_t})-\int_X \varphi_t \frac{d g_2(\chi_{\varphi_t})}{dt}\\
        &=-\int_{X}\varphi_t\sqrt{-1}\partial\bar{\partial}\frac{d\varphi_t}{dt}\wedge g_{2}'(\chi_{\varphi_t})\\
        &=-\int_{X}\frac{d\varphi_t}{dt}\sqrt{-1}\partial\bar{\partial}\varphi_t\wedge g_{2}'(\chi_{\varphi_t}),
    \end{split}
\end{equation}
and 
\begin{equation}
\label{eq:solution-based-first-variation}
    \begin{split}
        \frac{d}{dt}\mathcal{J}_{\varepsilon}(\varphi_t) =-\int_{X}\frac{d\varphi_t}{dt} g_{1,\varepsilon}(\chi_{\varphi_t}).
    \end{split}
\end{equation}
In particular, $g_{1,\varepsilon}(\chi)=g_1(\chi)=0$.
So $\int_X g_{1,\varepsilon}(\chi_\varphi)=0$.
Let $g_{1,\varepsilon}'(\chi_\varphi)$ be the formal derivative of $g_{1,\varepsilon}(\chi_\varphi)$,
and put $F_\varepsilon(\chi_\varphi)=\frac{g_{1,\varepsilon}(\chi_\varphi)}{g_2(\chi_\varphi)}$.

\begin{lem}\label{lem:solution-perturbation-ellipticity}
There exists $0<\varepsilon_0<1/n$ such that, for
$0<\varepsilon<\varepsilon_0$ and every $\varphi\in\mathcal H$,
\begin{equation}\label{eq:solution-perturbation-positivity}
 g_{1,\varepsilon}'(\chi_\varphi)
       -F_\varepsilon(\chi_\varphi)g_2'(\chi_\varphi) \ge 0,
 \qquad
 F_\varepsilon(\chi_\varphi)\le1-n\varepsilon.
\end{equation}
\end{lem}
\begin{proof}
By Lemma~\ref{lem:derivative-perturbation}, choose $\eta>0$ such that
$f_2\prec f_1+\eta f_2'$. Corollary~\ref{polarizationproperposition} and
Theorem~\ref{PositiveDerivative} show that for any $\alpha\in\Upsilon_{g_2}$,
\[
 F_\eta(\alpha)=\frac{\Pol_n(f_1+\eta f_2')(\alpha)}{\Pol_n(f_2)(\alpha)}=\frac{g_1(\alpha)+\eta\,\omega\wedge g_2'(\alpha)}{g_2(\alpha)}
\]
is nondecreasing when a nonnegative $(1,1)$-form is added to $\alpha$, since the eigenvalues are nondecreasing by the Courant-Fischer-Weyl min-max principle.

Choose $\varepsilon_0<\frac{1}{n}$ to be determined later. It suffices to prove the inequalities pointwise for every $0<\varepsilon<\varepsilon_0$ and $\alpha\in\Upsilon_{g_2}(\omega)$.

To prove the first inequality, it suffices to test its wedge product
with any test form $\theta=\sqrt{-1}\xi\wedge\bar\xi$ for a non-zero 1-form $\xi$, normalized so that $\theta\le\omega$. Since $\theta^2=0$, each of
$g_1(\alpha_t)$, $g_2(\alpha_t)$, and
$g_{1,\varepsilon}(\alpha_t)$ is affine in $t$, where $\alpha_t=\alpha+t\theta$. We write 
\[
g_2(\alpha_t)=g_2(\alpha)+t\theta\wedge g_2'(\alpha).
\]
Since $\theta\wedge g_2'(\alpha)>0$, there exists $t_0=\frac{-g_2(\alpha)}{\theta\wedge g_2'(\alpha)}$ such that $\alpha_t\in\Upsilon_{g_2}$ if and only if $t>t_0$. Then $g_2'(\alpha_{t_0})\ge 0$ by continuity.

By the monotonicity of $F_\eta$, $g_1(\alpha_{t_0})+\eta\,\omega\wedge g_2'(\alpha_{t_0})\le0$ since otherwise $F_\eta(\alpha_t)$ will converge to $\infty$ as $t\downarrow t_0$.
Choosing $C>0$ such that $\chi\le C\omega$ on $X$, we obtain
\[
 \begin{split}
 g_{1,\varepsilon}(\alpha_{t_0})
 &=g_1(\alpha_{t_0})-\varepsilon(\alpha_{t_0}-\chi)\wedge g_2'(\alpha_{t_0})\\
 &\le g_1(\alpha_{t_0})+\varepsilon C\,\omega\wedge g_2'(\alpha_{t_0})\\
 &\le-(\eta-\varepsilon C)\,\omega\wedge g_2'(\alpha_{t_0}) \le 0
 \end{split}
\] as long as we choose $\varepsilon_0<\frac{|\eta|}{C}$.
The affinity of $g_{1,\varepsilon}(\alpha+t\theta)$ and the definition of $t_0$ then imply
\[
 t_0\theta\wedge\bigl(g_{1,\varepsilon}'(\alpha)
       -F_\varepsilon(\alpha)g_2'(\alpha)\bigr)
 =g_{1,\varepsilon}(\alpha_{t_0})-g_{1,\varepsilon}(\alpha)
       +g_{1,\varepsilon}(\alpha)=g_{1,\varepsilon}(\alpha_{t_0}) \le 0.
\]
This proves the first inequality in \eqref{eq:solution-perturbation-positivity} since $t_0<0$.

For the upper bound, consider the pointwise ray $\alpha+t\omega \in \Upsilon_{g_2}$ for $t\ge 0$.
The first inequality in \eqref{eq:solution-perturbation-positivity} gives
\[
\frac{d}{dt}F_\varepsilon(\alpha+t\omega)
=\frac{\omega\wedge(g_{1,\varepsilon}'-F_\varepsilon g_2')(\alpha+t\omega)}
{g_2(\alpha+t\omega)} \ge 0.
\]
Both $g_1(\alpha)$ and $g_2(\alpha)$ have leading term $\alpha^n$,
whereas $(\alpha-\chi)\wedge g_2'(\alpha)$ has leading term
$n\alpha^n$. Therefore
$\lim_{t\to\infty}F_\varepsilon(\alpha+t\omega)=1-n\varepsilon>0$ by requiring $\varepsilon_0<\frac{1}{n}$. Monotonicity gives $F_\varepsilon(\alpha)\le1-n\varepsilon$, as required.
\end{proof}

\begin{prop}\label{prop:solution-implies-coercivity}
Fix $0<\varepsilon<\varepsilon_0$ as in
Lemma~\ref{lem:solution-perturbation-ellipticity}. Then
\begin{equation}\label{eq:perturbed-global-minimum}
 \mathcal J_\varepsilon(\varphi)\ge\mathcal J_\varepsilon(0)
 \quad\text{for every }\varphi\in\HH.
\end{equation}
Consequently, there exist $\delta>0$ and $C$ such that
$J_1(\varphi)\ge\delta d_1(0,\varphi)-C$ for every
$\varphi\in\HH$ with $I_2(\varphi)=0$.
\end{prop}
\begin{proof}
Fix an arbitrary $\varphi\in\HH$. We prove
\eqref{eq:perturbed-global-minimum} on smooth regularized
geodesics before passing to the limit.
For sufficiently small $\rho>0$, let $\varphi_{\rho,t}$,
$0\le t\le1$, be the smooth $\rho$-geodesic with
$\varphi_{\rho,0}=0$ and $\varphi_{\rho,1}=\varphi$, supplied by
\cite[Theorem 1.1]{ChenGhoshNieGeodesics}. The parameter $\rho$ regularizes the geodesic, whereas the perturbation size $\varepsilon$ is fixed.
Write $\chi_{\rho,t}=\chi+\ddc\varphi_{\rho,t}$.
By Lemma 3.2 and Proposition 2.21 of \cite{ChenGhoshNieGeodesics},
the geodesic equation can be written as:
\begin{equation}\label{eq:regularized-geodesic-error}
 \ddot\varphi_{\rho,t}\,g_2(\chi_{\rho,t})
 -\sqrt{-1}\partial\dot\varphi_{\rho,t}\wedge
       \bar\partial\dot\varphi_{\rho,t}\wedge
g_{2}'(\chi_{\rho,t})=-4\rho^2e^{2t}\widetilde g_2(\chi_{\rho,t})\ge 0,
\end{equation}
where $\widetilde g_2$ is the form defined in Lemma 3.2 of \cite{ChenGhoshNieGeodesics}.

Set $\mathcal J_\rho(t)=\mathcal J_\varepsilon(\varphi_{\rho,t})$.
The first-variation formula
\eqref{eq:solution-based-first-variation} and closedness of
$g_{1,\varepsilon}'(\chi_{\rho,t})$ give
\begin{equation}
    \begin{split}
        \mathcal J_{\rho}''(t)&=-\int_{X}\ddot\varphi_{\rho,t} g_{1,\varepsilon}(\chi_{\rho,t})-\int_{X}\dot\varphi_{\rho,t}g_{1,\varepsilon}'(\chi_{\rho,t})\wedge \sqrt{-1}\partial\bar{\partial}\dot\varphi_{\rho,t}\\
        &=-\int_{X}\ddot\varphi_{\rho,t} g_{1,\varepsilon}(\chi_{\rho,t})+\int_{X}\sqrt{-1}\partial(\dot\varphi_{\rho,t})\wedge\bar{\partial}(\dot\varphi_{\rho,t})\wedge g_{1,\varepsilon}'(\chi_{\rho,t}).
    \end{split}
\end{equation}

Substituting \eqref{eq:regularized-geodesic-error}, we obtain
\begin{equation}\label{eq:perturbed-second-variation}
 \begin{split}
 \mathcal J_\rho''(t)
 ={}&\int_X\sqrt{-1}\partial\dot\varphi_{\rho,t}\wedge
          \bar\partial\dot\varphi_{\rho,t}\wedge
          \bigl(g_{1,\varepsilon}'(\chi_{\rho,t})-F_{\varepsilon}(\chi_{\rho,t})g_{2}'(\chi_{\rho,t})
                    \bigr)
          +4\rho^2e^{2t} \int_X F_\varepsilon(\chi_{\rho,t})\widetilde g_2(\chi_{\rho,t})\\
 \ge{}&4\rho^2e^{2t}(1-n\varepsilon)\int_X \widetilde g_2(\chi_{\rho,t}).
 \end{split}
\end{equation}
 The first integral is nonnegative by Lemma \ref{lem:solution-perturbation-ellipticity}. For the second integral, $\int_X \widetilde g_2(\chi_{\rho,t})$ is independent of $\rho$ and $t$.

For every $\rho>0$, the initial derivative
\begin{equation}\label{eq:perturbed-initial-derivative}
 \mathcal J_\rho'(0)
 =-\int_X\dot\varphi_{\rho,0}\,g_{1,\varepsilon}(\chi)=0
\end{equation}
because $g_{1,\varepsilon}(\chi)=0$.
Integrating \eqref{eq:perturbed-second-variation} twice and
using \eqref{eq:perturbed-initial-derivative} gives
\[
 \mathcal J_\varepsilon(\varphi)-\mathcal J_\varepsilon(0)
 =\mathcal J_\rho'(0)+\int_0^1(1-t)\mathcal J_\rho''(t)\,dt\ge-(1-n\varepsilon)C\rho^2.
\]
Letting $\rho\to 0$ therefore proves $\mathcal J_\varepsilon(\varphi) \ge \mathcal J_\varepsilon(0)$ for all $\varphi \in \HH$.

Finally, if $I_2(\varphi)=0$, then Lemma~\ref{lem:based-energy-coercivity} gives
\[
J_1(\varphi)\ge\varepsilon\mathcal A(\varphi)-C
\ge\varepsilon a\,d_1(0,\varphi)-\varepsilon b-C.
\]
This finishes the proof.
\end{proof}

\section{Examples}

In this section, we show that interlacing together with strict proper position gives a sufficient condition for all the polynomial hypotheses in Theorem~\ref{thm:main}. The case $f_1(x)=x^{n-1}(x-1)$ and $f_2(x)=x^n$ recovers the results by Collins and Sz\'ekelyhidi \cite{CollinsSzekelyhidiJFlow}. Moreover, if $\hat\theta$ is in the supercritical phase $(\frac{n-2}{2}\pi, \frac{n}{2}\pi)$, then for any $\frac{n-2}{2}\pi<\theta_0<\hat\theta<\frac{n}{2}\pi$, the polynomials
\[
        f_1(x)=\Im(e^{-\sqrt{-1}\hat\theta}(1+\sqrt{-1}x)^n), \quad f_2(x)=\Im(e^{-\sqrt{-1}\theta_0}(1+\sqrt{-1}x)^n)
\]
also satisfy all of these polynomial hypotheses. This reduces to the Collins-Yau \cite{CollinsYauGeodesics} and Chu-Lee \cite{ChuLeeHypercriticalDHYM} case if $\hat\theta$ is in the hypercritical phase $(\frac{n-1}{2}\pi, \frac{n}{2}\pi)$ and $\theta_0=\hat\theta-\frac{\pi}{2}$.

\begin{prop}\label{prop:strict-interlacing}
Let $f_1,f_2$ be monic real polynomials of degree $n\ge2$, with real roots $b_1,\ldots,b_n$ and $a_1,\ldots,a_n$, respectively, counted with multiplicity. Suppose that
\begin{equation}\label{eq:weak-interlacing-order}
 a_1\le b_1\le a_2\le b_2\le\cdots\le a_n\le b_n,
\end{equation}
and that $f_2\prec f_1$ strictly. Then $f_1,f_2$ are right-Noetherian, $f_2\prec f_1$ strongly strictly, and $f_2\idealprec f_1$.
\end{prop}

\begin{proof}
Rolle's theorem shows that both polynomials and all their nonconstant derivatives have all real roots, and have nonincreasing largest roots. So both polynomials are right-Noetherian. By the remark after Definition 1.1 and Lemma 1.8 of \cite{BorceaBrandenLeeYangI}, the interlacing condition with the order \eqref{eq:weak-interlacing-order} is equivalent to the real stability of $f_2(x)y+f_1(x)$. By Lemma 2.4(f) of \cite{WagnerStablePolynomials}, for all $k=0,\ldots,n-1$, the real stability of $f_2(x)y+f_1(x)$ implies the real stability of $f_2^{(k)}(x)y+f_1^{(k)}(x)$, which is equivalent to the interlacing condition for $f_2^{(k)}(x)$ and $f_1^{(k)}(x)$ with the order \eqref{eq:weak-interlacing-order}. By Proposition 4.3 of \cite{FangMaIdealGarding}, for all $k=0,\ldots,n-1$, $f_2^{(k)}\idealprec f_1^{(k)}$, and hence $f_2^{(k)}\prec f_1^{(k)}$. By induction, it suffices to prove that $f_2'\prec f_1'$ strictly.

After cancelling the common roots of $f_1$ and $f_2$, we get $\frac{f_1}{f_2}=\frac{p}{q}$ for degree $d$ polynomials $p$, $q$ with all real roots, such that the roots of $p$, $q$ are simple and strictly interlace with the order \eqref{eq:weak-interlacing-order}. Thus,
\begin{equation}\label{eq:interlacing-partial-fractions}
 \frac{f_1(x)}{f_2(x)}=\frac{p(x)}{q(x)}
 =1+\sum_{q(\alpha)=0}\frac{A_\alpha}{x-\alpha},
 \qquad A_\alpha=\frac{p(\alpha)}{q'(\alpha)}<0.
\end{equation}
We can prove \eqref{eq:interlacing-partial-fractions} by the fact that $p(x)-(1+\sum_{q(\alpha)=0}\frac{A_\alpha}{x-\alpha})q(x)$ has degree at most $d-1$ and vanishes at every root of $q$.

Let $m$ be the multiplicity of $f_2$ at $r(f_2)=a_n$. If $m\ge 2$, then $r(f_2')=r(f_2)$, and Proposition~\ref{Strict proper position}(iii) gives $f_2'\prec f_1'$ strictly. If $m=1$, Rolle's theorem shows that $a_{n-1}<r(f_2')<a_n$, so $f_2(r(f_2'))<0$. Differentiating \eqref{eq:interlacing-partial-fractions} and using $f_2'(r(f_2'))=0$, we obtain
\[
f_1'(r(f_2'))=-f_2(r(f_2'))\sum_{q(\alpha)=0}\frac{A_\alpha}{(r(f_2')-\alpha)^2}<0.
\]
This implies $f_2'\prec f_1'$ strictly by Proposition~\ref{Strict proper position}(iii).
\end{proof}

\begingroup
\small
\bibliographystyle{plain}
\bibliography{reference}

\begin{center}
\begin{minipage}{0.88\textwidth}
\small
\textbf{Gao Chen}\\[0.35em]
University of Science and Technology of China\\
No.\ 96 Jinzhai Road, Baohe District, Hefei, Anhui 230000, P.\,R.\ China\\[0.35em]
\href{mailto:chengao1@ustc.edu.cn}{\texttt{chengao1@ustc.edu.cn}}\\
\textbf{Kartick Ghosh}\\[0.35em]
University of Science and Technology of China\\
No.\ 96 Jinzhai Road, Baohe District, Hefei, Anhui 230000, P.\,R.\ China\\[0.35em]
\href{mailto:kghosh@ustc.edu.cn}{\texttt{kghosh@ustc.edu.cn}}\\
\textbf{Ziyue Wang}\\[0.35em]
University of Science and Technology of China\\
No.\ 96 Jinzhai Road, Baohe District, Hefei, Anhui 230000, P.\,R.\ China\\[0.35em]
\href{mailto:wangziyue@mail.ustc.edu.cn}{\texttt{wangziyue@mail.ustc.edu.cn}}
\end{minipage}
\end{center}

\endgroup
\end{document}